\documentclass[12pt]{amsart}
\usepackage[utf8]{inputenc}
\usepackage[letterpaper, left=1.5in, right=1in, top=1in, bottom=1in]{geometry}
\usepackage{float} % Added for precise figure placement

\usepackage{amsfonts, amssymb, amsmath, bbm, latexsym, amsthm, amscd, mathrsfs, stmaryrd} 
\usepackage{amsthm}  %For theorems, lemmas and proofs
\usepackage{textcomp}
\usepackage{mathtools}
\usepackage{tikz}
\usepackage{tikz-cd}
\usepackage{pgfplots}
\pgfplotsset{compat=1.18, width = 10 cm} 
\usetikzlibrary{shapes,arrows,decorations.markings}
\usepackage{dyntree}
\usepackage{times} % Change to another font if needed
\usepackage{boldline} 
\usepackage{graphics, graphicx, subcaption} % This is used for adding images
\usepackage{hyperref}  %Referencing click
\usepackage{imakeidx}

\newcounter{myequation}[subsection]
\renewcommand{\themyequation}{\thesubsection.\alph{myequation}}

\newcommand{\myeqn}[1]{\refstepcounter{myequation}\tag{\themyequation} #1}

\usepackage{etoolbox}
\preto\subsection{\setcounter{myequation}{0}}

\newcommand{\suchthat}{\;\ifnum\currentgrouptype=16 \middle\fi|\;}

\newcommand{\heart}{\ensuremath\heartsuit}

\newcommand{\ds}{\displaystyle} 
\newcommand{\mb}{\mathbb}
\newcommand{\mcO}{ \mathcal{O} }
\newcommand{\bfF}{ {\bf F} }

\newcommand{\bfV}{ {\bf V} }
\newcommand{\bfG}{ {\bf G} }
\newcommand{\bfU}{ {\bf U} }
\newcommand{\bfI}{ {\bf I} }
\newcommand{\bfi}{ {\bf i}}
\newcommand{\bfW}{ {\bf W} }
\newcommand{\bfT}{ {\bf T} }
\newcommand{\bfS}{ {\bf S} }
\newcommand{\bfE}{ {\bf E} }
\newcommand{\bfip}{ {\bf i_+}}
\newcommand{\bfim}{ {\bf i_-}}
\newcommand{\bfj}{ {\bf j}}

\newcommand{\ev}{ {\bf E_{V, \Omega}}}
\newcommand{\evF}{ {\bf E_{V, \Omega}^{F}}}
\newcommand{\evh}{ {\bf E_{V, \Omega}^{\heart}}}
\newcommand{\evhF}{ {\bf E_{V, \Omega}^{\heart, F}}}
\newcommand{\evcF}{ {\bf E_{V, \Omega}^{c, F}}}
\newcommand{\evhat} { {\bf E_{\widehat{V}, \widehat{\Omega}}}}
\newcommand{\evhatF}{ {\bf E_{\widehat{V}, \widehat{\Omega}}^{\widehat{F}}}}

\newcommand{\hvO}{ {\bf H}_{\nu, \Omega}}
\newcommand{\hO}{ {\bf H}_{\Omega}}
\newcommand{\hvOh}{ {\bf H}_{\nu, \Omega}^{\heart}}
\newcommand{\hvOhat}{ {\bf H}_{\nu, \widehat{\Omega}}}
\newcommand{\hOh}{ {\bf H}_{\Omega}^{\heart}}
\newcommand{\hOhat}{ {\bf H}_{\widehat{\Omega}}}

\newcommand{\hvOc}{ {\bf H}_{\nu, \Omega}^c}
\newcommand{\hOc}{ {\bf H}_{\Omega}^c}
\newcommand{\hOIhat}{ {\bf H}_{\Omega}^{\widehat{I}}}

\newcommand{\ew}{ {\bf E_{W, \Omega}}}
\newcommand{\ewF}{ {\bf E_{W, \Omega}^{F}}}

\newcommand{\ewhF}{ {\bf E_{W, \Omega}^{\heart, F}}}
\newcommand{\ewcF}{ {\bf E_{W, \Omega}^{c, F}}}

\newcommand{\ewhatF}{ {\bf E_{\widehat{W}, \widehat{\Omega}}^{\widehat{F}}}}
\newcommand{\hwO}{ {\bf H}_{\omega, \Omega}}

\newcommand{\et}{ {\bf E_{T, \Omega}}}
\newcommand{\etF}{ {\bf E_{T, \Omega}^{F}}}

\newcommand{\etcF}{ {\bf E_{T, \Omega}^{c, F}}}
\newcommand{\ethF}{ {\bf E_{T, \Omega}^{\heart, F}}}

\newcommand{\ethatF}{ {\bf E_{\widehat{T}, \widehat{\Omega}}^{\widehat{F}}}}
\newcommand{\htO}{ {\bf H}_{\tau, \Omega}}

\newcommand{\gv}{ {\bf G_{V}}}
\newcommand{\gvF}{ {\bf G_{V}^F}}
\newcommand{\gvhat}{ {\bf G_{\widehat{V}}}}
\newcommand{\gvhatF} { {\bf G_{\widehat{V}}^{\widehat{F}}}}

\newcommand{\gt}{ {\bf G_{T}}}
\newcommand{\gtF}{ {\bf G_{T}^F}}

\newcommand{\gthatF} { {\bf G_{\widehat{T}}^{\widehat{F}}}}

\newcommand{\gw}{ {\bf G_{W}}}
\newcommand{\gwF}{ {\bf G_{W}^F}}

\newcommand{\gwhatF} { {\bf G_{\widehat{W}}^{\widehat{F}}}}

\newcommand{\la}{\langle}
\newcommand{\ra}{\rangle}
\newcommand{\whI}{\widehat{I}}
\newtheorem{theorem}{Theorem}[subsection] % reset theorem numbering for each subsection
\newtheorem{definition}[theorem]{Definition} % definition numbers are dependent on theorem numbers
\newtheorem{lemma}[theorem]{Lemma} % same for lemma numbers
\newtheorem{proposition}[theorem]{Proposition}
\newtheorem{example}[theorem]{Example}

\newtheorem{thrm}{Theorem}

\usepackage[english]{babel}
\usepackage{color}
\title[Hall algebras and edge contractions with loops]{Hall algebras and edge contractions with loops}

\author{Adhish Rele}
\address{Department of Mathematics\\ University at Buffalo\\ The State University of New York  \\Buffalo, NY 14260}
\email{adhishpr@buffalo.edu}
\date{\today}

\begin{document}
\pagenumbering{gobble}

\begin{abstract}
We extend the study of Hall algebras and edge contractions by generalizing Yiqiang Li’s work to contraction along vertices with multiple edges. Using the edge contractions, we establish new embeddings among Hall algebras in this broader setting. Our results demonstrate that these embeddings preserve key algebraic structures, including Hopf algebra operations. 
\end{abstract}

\maketitle

\section*{Introduction}
In 1900, while serving as a Privatdozent at Technische Hochschule Berlin Charlottenburg, Steinitz presented a talk at the Deutsche Mathematiker-Vereinigung meeting in Aachen \cite{Ste01}, where he introduced an algebra over the integers based on isomorphism classes of finite abelian groups, now known as the Hall algebra. He also proposed several conjectures that were later proven by Phillip Hall in the 1950s  \cite{Hall59} in the context of finite abelian p-groups. It has since emerged as a fundamental structure in representation theory, particularly in the study of quantum groups and geometric representation theory. The connection between Hall algebras and quantum groups was established by Ringel \cite{Rin90}, who extended the concept to more general abelian categories, such as the category of representations of quivers. He demonstrated that the composition algebra of the category of representations of a quiver, specifically of type ADE, can be viewed as a realization of the positive part of the quantum group associated with the underlying graph. Lusztig further expanded this perspective in \cite{Lus91}, providing a geometric approach to Hall algebras via perverse sheaves and canonical bases. These structures play a significant role in understanding quantum groups, categorification, and cluster algebras, leading to applications including, but not limited to, symplectic geometry, algebraic geometry and mathematical physics.\\

Edge contraction is a fundamental operation in graph theory where an edge is removed, and its incident vertices are merged into a single vertex. This operation has been extensively studied in various mathematical contexts, including topology, combinatorics, and representation theory. In the setting of Hall algebras and quantum groups, edge contraction provides insights into how algebraic structures associated with graphs behave under simplifications of the graph. Recent studies, including Li's solo work \cite{Li23} and his collaboration with Ren \cite{LR24}, have explored how Hall algebras and cohomological Hall algebras behave under edge contractions, demonstrating induced embeddings that inform the structure of quantum groups, among other algebraic structures.\\

This paper extends Li’s work on embeddings among Hall algebras via edge contractions by relaxing the assumption that there is exactly one edge connecting the vertices involved in the contraction. We generalize this setting to cases where multiple edges exist between these vertices, leading to the formation of loops, a new phenomenon, after contraction. This generalization requires revisiting fundamental structures, including the definition of the Cartan datum, quiver representation spaces, and Hall algebras. A central question explored in this work is whether the embedding phenomenon established in Li’s work persists under this broader framework. We prove the following key result:

\begin{thrm} [Theorem~\ref{main embedding}]
For a quiver with multiple edges between two vertices along which the contraction takes place, the edge contraction operation induces an embedding of the Hall algebra of the contracted quiver into the associated Hall algebra prior to the edge contraction. This embedding is compatible with the multiplication structure and preserves key algebraic properties such as the Hopf algebra structure.
\end{thrm}

It turns out that the embedding is a result of a composition of two injective algebra homomorphisms, namely $j_!$ and $\mu^\star$, described in section 2.3. The proof that these maps are algebra homomorphisms relies on the four lemmas in section 2.2. 

\begin{thrm} [Theorem~\ref{split subquotient}]
 Hall algebra of a contracted quiver is a split subquotient of the Hall algebra prior to the edge contraction.
\end{thrm}

The proof of this theorem follows directly from \cite{Li23}, Proposition 2.4.1 on page 17.\\

Looking ahead, our work opens up several exciting directions for further research. One promising avenue is to investigate the subquotient structure predicted by Li within the broader setting of Hall algebras that include loops. We aim to determine whether this phenomenon is a universal feature that can be observed across various algebraic systems. Additionally, exploring the interplay between edge contraction and modern categorification frameworks—such as Khovanov-Lauda-Rouquier algebras \cite{Mak18} and quiver Hecke algebras—could provide deeper insights into both the combinatorial and geometric aspects of representation theory. By clarifying these connections, we hope to pave the way for new applications in mathematical physics and a more unified understanding of the underlying algebraic structures.\\

This paper starts by defining the generalized Cartan datum, introducing key modifications necessary to accommodate quivers with multiple edges and loops. It also provides an overview of root data and Weyl groups associated with such quivers. Following this, we develop the representation spaces corresponding to quivers with loops, laying the groundwork for our study of Hall algebras in this setting. We then review Lusztig’s induction and restriction diagrams, which play a crucial role in constructing and understanding embeddings among Hall algebras. Using these tools, we define Hall algebras associated with quivers containing loops and establish our main result: the embedding of Hall algebras under edge contractions in this generalized framework. Finally, we demonstrate the compatibility of this embedding with the PBW basis and analyze its interaction with split short exact sequences\\

 \newpage

\subsection*{Acknowledgements}
I was first introduced to Hall algebras by Yiqiang Li when he shared his paper on Quantum Groups and Edge Contractions with me during my time as his student. Shortly thereafter, he became my research advisor and encouraged me to generalize his work to graphs with multiple edges between pairs of vertices. This work, which forms a part of my thesis, is the result of his unwavering encouragement, patience, and insightful guidance, which have been instrumental in bringing this project to fruition.\\

\pagenumbering{roman}
\tableofcontents

\newpage

\pagenumbering{arabic}
\section{Foundations}

In this section, we define the generalized Cartan data, study quivers and Weyl groups under edge contractions.\\

\subsection{Generalized Cartan data}

Define the set of natural numbers as $\mb{N} = \{0, 1, 2, .....\}$. The generalized Cartan datum is defined to be a quadruple $(I,  \cdot, \phi_1,  \phi_2)$, where $I$ is a finite set,  $\phi_1: I \longrightarrow \mb{N}\backslash \{0\}, ~\phi_2: I \longrightarrow \mb{N}$, along with a symmetric bilinear form ``$\cdot$" that maps values in $\mathbb{Z}$ and is defined on the free abelian group $\mathbb{Z}[I]$, meeting the following conditions:\\
1. $\text{~For every~} i \in I, \text{~we have,~}$ 
\[ i\cdot i = 2(\phi_1(i) - \phi_1(i)\phi_2(i)) \in 2\mb{Z}; \myeqn{} \label{cond1} \] \\
2. $\text{For distinct elements~} i, j \in I,$
\[ \ds\frac{i\cdot j}{\phi_1(i)} \in -\mb{N}.
\myeqn{} \label{cond2} \] \\

For a directed graph, the maps $\phi_1(i)$ and $\phi_2(i)$ count the number of vertices in the $a-$orbit [$i$] and the number of loops on $i$ respectively. See Section \ref{Graph Structures} for further details. The generalized Cartan datum restricts to Lusztig's Cartan datum when $\phi_2$ is identically equal to 0. \\

\begin{example}
Here is an example of a generalized Cartan datum for $I = \{a, b, c\}$. \\
\begin{equation}
\begin{aligned}  
\phi_1(a) &= 2 \hspace{1 cm} \phi_2(a) &= 2 \hspace{1 cm} a \cdot b &= -6\\
\phi_1(b) &= 3 \hspace{1 cm} \phi_2(b) &= 1 \hspace{1 cm} a \cdot c &= 0\\
\phi_1(c) &= 1 \hspace{1 cm} \phi_2(c) &= 3 \hspace{1 cm} b \cdot c &= -3
\end{aligned}
\myeqn{} \label{cartan example}
\end{equation}
\label{example 1}
\end{example}

Figure \ref{Example graph of a Cartan datum} in section \ref{Graph Structures} is an example of a graph whose Cartan datum would be the one in the above example. \\

Now, suppose there exists a particular pair $(i_+, i_-) \in I$ that satisfies the following:
\begin{equation}
\begin{aligned}  
\phi_1(i_+) &= \phi_1(i_-), \\
  \phi_2(i_+) &= \phi_2(i_-) = 0,\\
 i_+ \cdot i_- &\neq 0.
\end{aligned}
\myeqn{} \label{mainassumption}
\end{equation}

Let us fix such a pair.  Define $$i_0 = i_+ + i_- \in \mb{Z}[I]$$ 

Next, construct the set $\hat{I} = I \cup \{i_0\} - \{i_+,  i_-\}$. Note that the free abelian group $\mb{Z}[\hat{I}]$ forms a subgroup of $\mb{Z}[I]$ and the bilinear form $\cdot$ on $\mb{Z}[I]$ naturally restricts to a symmetric bilinear form on $\mb{Z}[\hat{I}],$ still denoted by $\cdot$.\\

With these definitions, we obtain the following lemma:

\begin{lemma} The quadruple $(\hat{I},  \cdot,  \widehat{\phi_1},  \widehat{\phi_2})$ is a generalized Cartan datum where \\ 
$\widehat{\phi_1}: \hat{I} \longrightarrow \mb{N}\backslash \{0\}$ is defined as $\widehat{\phi_1}(i) = \left\{ \begin{array} {c c} \phi_1(i)  & \text{if} \ \ i \neq i_0\\
\phi_1(i_+) ~(= \phi_1(i_-))~~~ & \text{if} \ \  i = i_0 \end{array} \right. $ \end{lemma}

\vspace{0.4 cm}
\noindent and $\widehat{\phi_2}: \hat{I} \longrightarrow \mb{N}$ is defined as $\widehat{\phi_2}(i) = \left\{ \begin{array} {c c} \phi_2(i)  & \text{if} \ \ i \neq i_0\\
\ds\frac{-i_+ \cdot i_-}{\phi_1(i_+)} - 1 ~~~ & \text{if} \ \  i = i_0 \end{array} \right. $.\\
\begin{proof}
We only need to show the condition (1) holds when $i_0$ is involved. First to show that $i_0\cdot i_0 = 2(\widehat{\phi_1}(i_0) - \widehat{\phi_1}(i_0)\widehat{\phi_2}(i_0))$. The left hand side of the equation is evaluated as follows: 
\begin{equation*}
\begin{aligned}
i_0\cdot i_0  &= (i_+ + i_-) \cdot (i_+ + i_-)\\ 
&= i_+\cdot i_+  + 2i_+\cdot i_- + i_-\cdot i_- \\
&= 2\phi_1(i_+) + 2i_+\cdot i_- + 2\phi_1(i_-).\\
\end{aligned}
\end{equation*}
And the right hand side is 
\begin{equation*}
\begin{aligned}
2[\widehat{\phi_1}(i_0) - \widehat{\phi_1}(i_0)\widehat{\phi_2}(i_0)] &=   2 \left[ \phi_1(i_+) - \phi_1(i_+)\left(  \ds\frac{-i_+ \cdot i_-}{\phi_1(i_+)} - 1 \right)  \right] \\ 
& = 2\phi_1(i_+) + 2i_+\cdot i_- + 2\phi_1(i_+)\\
& = 2\phi_1(i_+) + 2i_+\cdot i_- + 2\phi_1(i_-).\\
\end{aligned}
\end{equation*}
Hence, $i_0\cdot i_0 = 2(\widehat{\phi_1}(i_0) - \widehat{\phi_1}(i_0)\widehat{\phi_2}(i_0))$. Also, we have,  
\begin{equation*}
\begin{aligned}
\dfrac{i_0 \cdot j  }{\phi_1(i_0)} &= \dfrac{ (i_+ + i_-)\cdot j }{\phi_1(i_0)}  =  \dfrac{ i_+ \cdot j  }{\phi_1(i_+)} + \dfrac{ i_- \cdot j  }{\phi_1(i_-)} \in -\mb{N} ~~~ \forall j \in \hat{I} - \{i_0\},\\
\hspace{2.6 cm} \dfrac{j \cdot i_0  }{\phi_1(j)} &= \dfrac{ j \cdot (i_+ + i_-) }{\phi_1(i_0)}  =  \dfrac{j \cdot i_+  }{\phi_1(j)} + \dfrac{ j \cdot i_-  }{\phi_1(j)} \in -\mb{N} ~~~ \forall j \in \hat{I} - \{i_0\}.
\end{aligned}
\end{equation*}
This proves the lemma.
\end{proof}

The generalized Cartan datum $(\hat{I},  \cdot,  \widehat{\phi_1},  \widehat{\phi_2})$ is called the \underline{edge contraction} of $(I,  \cdot, \phi_1,  \phi_2)$ along the pair $\{i_+ ,   i_-\}$.\\

\subsection{Graph Structures} \label{Graph Structures}

Consider a finite oriented graph that may include loops, represented by the quadruple $({\bf I},  \Omega, ~ ' : \Omega \to {\bf I}, ~ '' : \Omega \rightarrow {\bf I}),  $ where ${\bf I}$ and $\Omega$ are two finite sets with ${\bf I} \neq \emptyset$. We define an admissible automorphism $a$ for this oriented graph as a pair of bijections, $(a: {\bf I} \rightarrow \bfI,  a: \Omega \rightarrow \Omega)$, satisfying:
\begin{itemize}
    \item $a(h') = a(h)',  a(h'') = a(h)'',  ~\forall h \in \Omega$,
    \item $\{a(h)' , a(h)''\} \nsubseteq [ \bfi]$ for all $h \in \Omega$ with $h' \neq h''$ and for any $\bfi \in \bfI$ where $[\bfi]$ is the $a$-orbit of $\bfi$.
\end{itemize}
Henceforth, $[\bfi]$ is replaced by $\bfi$ where the context is clear.\\

With the oriented graph $({\bf I},  \Omega)$ and its admissible automorphism $a$, we can derive a generalized Cartan datum $(\bfI /a,  \cdot, \phi_1,  \phi_2)$ where $\bfI / a $ represents the set of $a$-orbits in $\bfI$, and where:
\begin{enumerate}
    \item  $\phi_1:\bfI / a\rightarrow \mb{N} \backslash \{0\}$ defined as $\phi_1(\bfi) = \# [\bfi]$.
    \item $ \phi_2: \bfI /a \rightarrow \mb{N}$ is defined as $\phi_2(\bfi) = \# \{ h \in \Omega ~|~ h' = h'' = j   \text{~for some fixed~} j \in [\bfi]\}$.
    \item $[\bfi]\cdot [\bfi] = 2(\phi_1(\bfi) - \phi_1(\bfi)\phi_2(\bfi)) \in 2\mathbb{Z}$ for all $[\bfi] \in \bfI / a$.
    \item $[\bfi]\cdot [\bfj] = - \# \{h \in \Omega ~|~ h',  h'' \in [\bfi] \cup [\bfj],  h' \neq h'' \} $  for all $[\bfi] \neq [\bfj] \in \bfI / a$.
\end{enumerate}

It is evident that $\ds\frac{[\bfi] \cdot [\bfj]}{\phi_1(\bfi)} \in -\mb{N} ~~\forall [\bfi] \neq [\bfj] \in \bfI / a$.\\

Example of an oriented graph based on the Cartan datum in example \ref{example 1} with loops respecting the action of an admissible automorphism $a$ is as follows:

\begin{figure}[H]
  \centering
  \begin{tikzpicture}
    % Your first TikZ code here
 % First Ellipse
  \node[ellipse, minimum width=1.8cm, minimum height=4.8cm, draw] (ellipse1) at (0,0) {};
  \node[circle, fill, inner sep=1pt] (dot11) at (0, 0.7) {};
  \node[circle, fill, inner sep=1pt] (dot12) at (0, -0.7) {};
  
  % Second Ellipse
  \node[ellipse, minimum width=1.8cm, minimum height=4.8cm, draw] (ellipse2) at (4,0) {};
  \node[circle, fill, inner sep=1pt] (dot21) at (4, 1.2) {};
   \node[circle, fill, inner sep=1pt] (dot22) at (4, 0) {};
  \node[circle, fill, inner sep=1pt] (dot23) at (4, -1.2) {};

  % Third Ellipse
  \node[ellipse, minimum width=1.8cm, minimum height=4.8cm, draw] (ellipse3) at (8,0) {};
  \node[circle, fill, inner sep=1pt] (dot31) at (8, 0) {};

  \node at (0, 3) {\( [a] \)}; % Label for the first ellipse
  \node at (4, 3) {\( [b] \)}; % Label for the first ellipse
  \node at (8, 3) {\( [c] \)}; % Label for the first ellipse

  % Edges
  \draw[blue, -,postaction={decorate,decoration={markings,mark=at position 0.5 with {\arrow[thick]{stealth};}}}] (dot11) -- (dot21);
  \draw[blue, -,postaction={decorate,decoration={markings,mark=at position 0.5 with {\arrow[thick]{stealth};}}}] (dot11) -- (dot22);
  \draw[blue, -,postaction={decorate,decoration={markings,mark=at position 0.5 with {\arrow[thick]{stealth};}}}] (dot11) -- (dot23);
  \draw[blue, -,postaction={decorate,decoration={markings,mark=at position 0.5 with {\arrow[thick]{stealth};}}}] (dot12) -- (dot21);
  \draw[blue, -,postaction={decorate,decoration={markings,mark=at position 0.5 with {\arrow[thick]{stealth};}}}] (dot12) -- (dot22);
  \draw[blue, -,postaction={decorate,decoration={markings,mark=at position 0.5 with {\arrow[thick]{stealth};}}}] (dot12) -- (dot23);

  \draw[blue, -,postaction={decorate,decoration={markings,mark=at position 0.5 with {\arrow[thick]{stealth};}}}] (dot21) -- (dot31);
  \draw[blue, -,postaction={decorate,decoration={markings,mark=at position 0.5 with {\arrow[thick]{stealth};}}}] (dot22) -- (dot31);
  \draw[blue, -,postaction={decorate,decoration={markings,mark=at position 0.5 with {\arrow[thick]{stealth};}}}] (dot23) -- (dot31);

    %Loops
  \draw[red, -,postaction={decorate,decoration={markings,mark=at position 0.5 with {\arrow[thick]{stealth};}}}] (dot11) to[out=150, in=110, looseness=170] (dot11);
  \draw[red, -,postaction={decorate,decoration={markings,mark=at position 0.5 with {\arrow[thick]{stealth};}}}] (dot11) to[out=200, in=150, looseness=120] (dot11);
  \draw[red, -,postaction={decorate,decoration={markings,mark=at position 0.5 with {\arrow[thick]{stealth};}}}] (dot12) to[out=180, in=140, looseness=170] (dot12);
  \draw[red, -,postaction={decorate,decoration={markings,mark=at position 0.5 with {\arrow[thick]{stealth};}}}] (dot12) to[out=230, in=180, looseness=120] (dot12);

  \draw[red, -,postaction={decorate,decoration={markings,mark=at position 0.5 with {\arrow[thick]{stealth};}}}] (dot21) to[out=80, in=50, looseness=110] (dot21);
  \draw[red, -,postaction={decorate,decoration={markings,mark=at position 0.5 with {\arrow[thick]{stealth};}}}] (dot22) to[out=80, in=50, looseness=110] (dot22);
  \draw[red, -,postaction={decorate,decoration={markings,mark=at position 0.5 with {\arrow[thick]{stealth};}}}] (dot23) to[out=80, in=50, looseness=110] (dot23);

  \draw[red, -,postaction={decorate,decoration={markings,mark=at position 0.5 with {\arrow[thick]{stealth};}}}] (dot31) to[out=80, in=30, looseness=140] (dot31);
  \draw[red, -,postaction={decorate,decoration={markings,mark=at position 0.5 with {\arrow[thick]{stealth};}}}] (dot31) to[out=30, in=-20, looseness=140] (dot31);
  \draw[red, -,postaction={decorate,decoration={markings,mark=at position 0.5 with {\arrow[thick]{stealth};}}}] (dot31) to[out=-20, in=-70, looseness=140] (dot31);

  \end{tikzpicture}

\caption{  { \it An example of a graph with 6 vertices and 3 a-orbits drawn as ellipses. The red edges are loops and the blue edges are the between different vertices } }
\label{Example graph of a Cartan datum}
\end{figure}
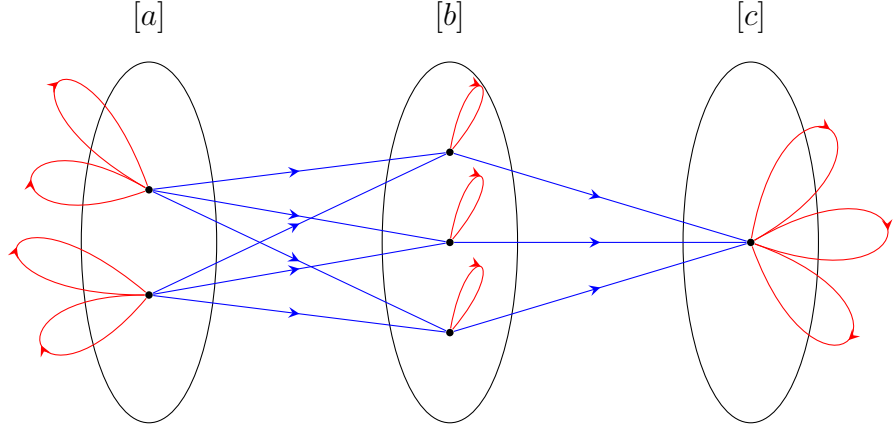

Assume that there exist two orbits $[\bfip] , [\bfim] \in \bfI /a$ such that 
\begin{equation}
\begin{aligned}
1. \quad & \text{\#} [\bfip] = \text{\#} [\bfim] \text{ i.e. }  \phi_1(\bfip) = \phi_1(\bfim). \\
2. \quad & \text{If } h \in \Omega \text{ such that } h', h'' \in [\bfip] \cup [\bfim],  h' \neq h'' \text{ then there does not exist} \\ 
         & l \in \Omega \text{ such that } l', l'' \in [\bfip] \cup [\bfim],  l' \in \{h',  h'' \} \text{ but } I'' \notin \{h', h'' \}. \\
3. \quad & \phi_2(\bfip) = \phi_2(\bfim) = 0. \\
\end{aligned}
\myeqn{} \label{graphassumption}
\end{equation}

Fix the pair $[\bfip] , [\bfim] \in \bfI /a$.  Choose an edge $e \in \Omega$ such that $e' \in [\bfip]$ and $e'' \in [\bfim]$. If there exists no such edge, rename $[\bf i_+]$ to $[\bfim]$ and vice versa. Define $h_k = a^k(e) ~~~\forall~ 1 \leq k \leq \phi_1([\bfi])$. Now, for any edge $h \in \Omega$, let $\overline{h}$ be a corresponding formal symbol that will represent the edge opposite to $h$, i.e., $h^{\prime} = \overline{h}^{\prime\prime}$ and $h^{\prime\prime} = \overline{h}^{\prime}$. We construct a new oriented graph as follows:
\begin{equation*}
\begin{aligned}
\widehat{\bfI} &= \bfI - [\bfim],\\
\widehat{\Omega} &= \Omega ~~\cup~~ \{ l_1 h_k  ~|~ h_k'' = l_1'  \}_{\{ 1 \leq k \leq \phi_1([\bfip])  \}}~~ \cup~~ \{ \overline{h_k}~l_2  ~|~ h_k'' = l_2''  \}_{\{ 1 \leq k \leq \phi_1([\bfip])  \}} \\
 & \hspace{0.4 cm} - \{ h ~|~ h', h'' \in [\bfim] \} .
\end{aligned}
\end{equation*}
The new maps $^{\prime}: \widehat{\Omega} \longrightarrow \widehat{I}$ and $^{\prime\prime}: \widehat{\Omega} \longrightarrow \widehat{I}$ on ($\widehat{I}, \widehat{\Omega}$) are induced by the original maps $^{\prime}$ and $^{\prime\prime}$ on the graph ($I, \Omega$) and by defining
$$(l_1h_k)^{\prime} = h_k^{\prime},~ (l_1h_k)^{\prime\prime} = l_1^{\prime\prime} \text{~whenever~} h_k'' = l_1',$$
$$( \overline{h_k}~l_2)^{\prime} = l_2^{\prime}, ~ ( \overline{h_k}~l_2)^{\prime\prime} = h_k^{\prime}  \text{~whenever~} h_k'' = l_1''.$$
We further define an automorphism $\widehat{a}$ on the new graph ($\widehat{I}, \widehat{\Omega}$) by defining $\widehat{a}: \widehat{I} \longrightarrow \widehat{I}$ as the restriction of $a: I \longrightarrow I$ and $a: \widehat{\Omega} \longrightarrow \widehat{\Omega}$ as follows:
\[
\begin{aligned}
\widehat{a}(h) &= a(h) & \text{if~~} h \in \Omega, \left.\hspace{0.85 cm} \right. \\
\widehat{a}(l_1h_k) &= a(l_1)a(h_k)  & \text{if~~} h_k'' = l_1' \in \Omega, \\
\widehat{a}(\overline{h_k}l_2) &= \overline{a(h_k)}a(l_2)  & \text{if~~} h_k'' = l_2'' \in \Omega .
\end{aligned} \]

\begin{lemma} The automorphism $\widehat{a}$ is admissible. \end{lemma}  
\begin{proof}
    
It is clear that for any edge $h \in \Omega$,~ $\widehat{a}(h^{\prime}) = \widehat{a}(h)',~ \widehat{a}(h^{\prime\prime}) = \widehat{a}(h)''$ and $\{\widehat{a}(h)',~ \widehat{a}(h)''\} \nsubseteq [\bfi]$ for any $[\bfi] \in {\bf \widehat{I}} \backslash \widehat{a}$. Consider the edge $l_1h_k \in \widehat{\Omega}, ~~h_k'' = l_1'$. To show $\widehat{a}(l_1h_k)' = \widehat{a}((l_1h_k)')$. Notice that $(\widehat{a}(l_1h_k))' = (a(l_1)a(h_k))' = a(h_k)' = a(h_k')$ and $\widehat{a}((l_1h_k)') = \widehat{a}(h_k') = a(h_k')$.\\
Similarly, one can show that $\widehat{a}(l_1h_k)'' = \widehat{a}((l_1h_k)'')$. 
If $l_1'' = h_k'$ then there is nothing to prove since that will imply $(l_1h_k)' = (l_1h_k)''$.
Finally we show that $\widehat{a}(l_1h_k)'$ and $\widehat{a}(l_1h_k)''$ are in different $\widehat{a}-$orbits.
Suppose $l_1'' \neq h_k'$. By assumption (2) in Section 1.2, $l_1'' \notin [\bfip]$. From above, $\widehat{a}(l_1h_k)' = a(h_k')$ and $\widehat{a}(l_1h_k)'' = a(l_1'')$. Since $h_k'$ and $l_1''$ are in different $a-$orbits, we can conclude that $\widehat{a}(l_1h_k)'$ and $\widehat{a}(l_1h_k)''$ are in different $\widehat{a}-$orbits. Hence, the automorphism $\widehat{a}$ is admissible.
\end{proof}

\begin{proposition} Let $(I,  \cdot, \phi_1,  \phi_2)$ be a generalized Cartan datum. There exists a finite graph $(\bfI,  H,  ',  '')$ and an admissible automorphism $a$ of this graph such that  $(I,  \cdot, \phi_1,  \phi_2)$ is obtained from them by the construction in the second paragraph in Section 1.2. \end{proposition}
\begin{proof}
For each $i \in I$,  we consider the set $D_i$ such that $|D_i| = \phi_1(i)$ and a cyclic permutation $a: D_i \rightarrow D_i$.  Let $\bfI = \sqcup_{i \in I} D_i$ and let $a: \bfI \rightarrow \bfI$ be the permutation whose restriction to each $D_i$ is the permutation $a: D_i \rightarrow D_i$ considered above. 
The proof in the case when there are no loops can be found in \cite{Lus10} (Chap. 14). There are certain additions to be made in the case of graphs with loops and they are presented below. 

Choose $i \in I$ such that $\phi_2(i) \neq 0$. Let $K = \{( \alpha,  \alpha) ~|~ \alpha \in D_i\}$. Consider $a \times a: K \longrightarrow K$. Let $\rho_1$ be the $a$-orbit. This implies $|\rho_1|  = |D_i| = \phi_1(i)$. Let $H_i  = \displaystyle\sqcup_{\phi_2(i)} \rho_1$. Hence,  $|H_i| = \phi_2(i)\phi_1(i)$. We have,
\begin{equation*}
\begin{aligned} 
\bfI / a &= \{D_i ~|~ i \in I\},\\
D_i \cdot D_i &= 2(|D_i|  - |H_i| ) \\
&= 2(\phi_1(i) - \phi_1(i)\phi_2(i)) \\
&= i\cdot i,\\
D_i \cdot D_j &= -\{h \in H ~|~ h' , h'' \in \{D_i ,  D_j\} \} \\
&= - |H_{ij}| \\
&= - \dfrac{-i\cdot j}{{\text{lcm}(|D_i|, |D_j|  )}} \text{lcm}(|D_i|, |D_j|  ) \\
&= i \cdot j ~~ (i \neq j) .
\end{aligned}
\end{equation*}
where $H_{ij} = \ds\bigsqcup_{\frac{-i\cdot j}{\text{lcm~}(\phi_1(i), \phi_1(j))}} \rho$ ~~~~where $\rho$ is an orbit of the permutation $a: D_i \times D_j \longrightarrow D_i \times D_j$\\

Let $\widehat{\phi_1}$ and $\widehat{\phi_2}$ be the maps obtained from the constructed graph. 
To show that $\widehat{\phi_1} = \phi_1$ and $\widehat{\phi_2} = \phi_2$. We can see that,
\begin{equation*}
\begin{aligned}
\widehat{\phi_1}(i) &= \#[i] = |D_i|  = \phi_1(i)
\widehat{\phi_2}(i) &= \dfrac{|H_i|}{\widehat{\phi_1(i)}} = \dfrac{\phi_1(i)\phi_2(i)}{\phi_1(i)} = \phi_2(i).
\end{aligned}
\end{equation*}
This proves the proposition.
\end{proof} 

\begin{definition}
A Cartan datum $(I,  \cdot, \phi_1,  \phi_2)$ is said to be isomorphic to another Cartan datum $(J, \circ, \widehat{\phi_1}, \widehat{\phi_2})$ if there exists a bijection from $I$ to $J$ preserving the bilinear forms, $\phi_1 = \widehat{\phi_1}$ and $\phi_2 = \widehat{\phi_2}$. 
\end{definition}

\begin{lemma} Consider a graph $({\bf I},  \Omega, ~ ' : \Omega \rightarrow {\bf I}, ~ '' : \Omega \rightarrow {\bf I})$ with the corresponding generalized Cartan datum $(\bfI /a,  \cdot, \phi_1,  \phi_2)$. Then the edge contraction of the above Cartan datum $(\tilde{\bfI} /\tilde{a},  \cdot, \widetilde{\phi_1}, 
\widetilde{\phi_2)}$ is isomorphic to the Cartan datum $(\hat{\bfI} /\hat{a},  \circ, \widehat{\phi_1}, 
\widehat{\phi_2)}$ of the graph obtained from an edge contraction of the above graph.  
\end{lemma}  
\begin{proof}
 First we need to check that $[i_+] \circ [j] = [i_0] \cdot j$ for every $[j] \neq [i_+], [j] \neq [i_0]$. 
Now by definition, 
\begin{equation*}
\begin{aligned}
[i_+] \circ [j] &= - \# \{h \in \widehat{\Omega} ~|~ h',  h'' \in [i_+] \cup [j],  h' \neq h'' \}\\
& = - \# \{h \in \Omega ~|~ h',  h'' \in [i_+] \cup [j],  h' \neq h'' \} - \# \{h \in \Omega ~|~ h',  h'' \in [i_-] \cup [j],  h' \neq h'' \} \\
&= [i_+] \cdot [j] + [i_-] \cdot [j]\\
 &= [i_0] \cdot [j]
\end{aligned}
\end{equation*}
Next, $\widehat{\phi_1}([i_+]) = \# [i_+] = \phi_1([i_+]) = \widetilde{\phi_1}([i_0])$.\\
Also, $\widehat{\phi_2}([i_+]) = \# \{ h \in \Omega ~|~ h' = h'' = j   \text{~for some fixed~} j \in [i_+]\} = \dfrac{-[i_+]\cdot[i_-]}{\phi_1(i_+)} - 1 = \widetilde{\phi_2}(i_0)$. The lemma is proved.
\end{proof}

\subsection{Root datum}

We define root datum associated with a generalized Cartan datum. 

\begin{definition} Root datum associated with $(I, \cdot, \phi_1, \phi_2)$ consists of:\\
1) Finitely generated free abelian groups $Y$ and $X$\\
2) A perfect bilinear pairing $\la \cdot, \cdot\ra: Y \times X \longrightarrow \mb{Z}$\\
3) An embedding $I \hookrightarrow X ~(i \mapsto i')$ and an embedding $I \hookrightarrow Y ~(i \mapsto i)$\\
4) $\la i, j'\ra = \frac{i\cdot j}{\phi_1(i)}~~\forall i, j \in I$
\end{definition}

The elements of $Y$ are called roots and the elements of $X$ are called co-roots. Let $(\widehat{I}, \cdot, \widehat{\phi_1}, \widehat{\phi_2})$ be the edge contraction of $(I, \cdot, \phi_1, \phi_2)$ along the pair $(i_+, i_-)$. 
Embedding $I \hookrightarrow X$ defines a group homomorphism $\mb{Z}[I] \longrightarrow X$. Define an embedding $\widehat{I} \hookrightarrow X$ to be the restriction of the homomorphism $\mb{Z}[I] \longrightarrow X$ to $\widehat{I}$, i.e. $i \mapsto i'$ if $i \in \whI \backslash {i_0}$ and $i_0 \mapsto (i_+)' + (i_-)'$. 
Embedding $I \hookrightarrow Y$ defines a group homomorphism $\mb{Z}[I] \longrightarrow Y$. Define an embedding $\widehat{I} \hookrightarrow Y$ to be the restriction of the homomorphism $\mb{Z}[I] \longrightarrow Y$ to $\widehat{I}$, i.e. $i \mapsto i$ if $i \in \whI \backslash {i_0}$ and $i_0 \mapsto i_+ + i_-$. It is easy to check that the collection ($Y, X, \la \cdot, \cdot \ra , \widehat{I} \hookrightarrow Y, \widehat{I} \hookrightarrow X $) is a root datum associated to the Cartan datum ($\whI, \cdot$).\\

A root datum as above is called $X-$regular (resp. $Y-$regular) if the image of the embedding $I \subset X$ (resp. $I \subset Y$) is linearly independent in X (resp. Y). A $Y-$regular root datum is called simply connected root datum and a $X-$regular root datum is called adjoint root datum.

\subsection{Weyl Groups}
Let ($Y, X, \la \cdot, \cdot \ra , I \hookrightarrow Y, I \hookrightarrow X $) be a simply connected root datum associated to $(I, \cdot)$. We define a reflection $s_i: Y \rightarrow Y$ by $s_i(y) = y - \la y, i' \ra i$ for each $i \in I$. The Weyl group is a group generated by these reflections and is denoted by $W_I$. We call $W_I$ to be the Weyl group associated to the Cartan datum $(I, \cdot, \phi_1, \phi_2)$. Let $(\whI, \cdot, \widehat{\phi_1}, \widehat{\phi_2})$ be the edge contraction of $(I, \cdot, \phi_1, \phi_2)$ along the pair $\{i_+, i_-\}$. Let $W_{\whI}$ be the Weyl group associated to $(\whI, \cdot, \widehat{\phi_1}, \widehat{\phi_2})$. 

\begin{lemma}
The function $\psi: W_{\whI} \longrightarrow \text{~Aut}(Y)$ defined as $s_i \mapsto s_i$ for all $i \in \widehat{I} - {i_0}$ and $s_{i_0} \mapsto s_{i_+}s_{i_- + \widehat{\phi_2}(i_0)i_+} s_{i_+}$ defines a group embedding where $s_{i_- + \widehat{\phi_2}(i_0)i_+}(y) = y - \la y, i'_- + \widehat{\phi_2}(i_0)i'_+  \ra (i_- + \widehat{\phi_2}(i_0)i_+)$. \end{lemma}  
\begin{proof}
 We only need to show that $s_{i_0} = s_{i_+}s_{i_- + \widehat{\phi_2}(i_0)i_+} s_{i_+}$. We divide the proof in 5 steps:\\ 
\underline {Step 1}
\begin{equation*}
\begin{aligned}
s_{i_-}(i_+) &= i_+ - \la i_+, i_-' \ra i_- \\ 
&= i_+ - \frac{i_+\cdot i_-}{\phi_1(i_+)} i_- \\
&= i_+ + (\widehat{\phi_2}(i_0) + 1)i_- \\
&= i_0 + \widehat{\phi_2}(i_0)i_- .\\
\end{aligned}
\end{equation*}
\underline {Step 2}
\begin{center}
 Similarly, we can show,  $s_{i_+}(i_-) = i_0 + \widehat{\phi_2}(i_0)i_+$. \\
\end{center}
\underline {Step 3} 
\begin{equation*}
\begin{aligned}
s_{i_+}(i_0) &= i_0 - \la i_0, i_+' \ra i_+ \\
&= i_0 - \frac{i_0 \cdot i_+}{\phi_1(i_0)} i_+ \\
&= i_0 - (1 - \widehat{\phi_2}(i_0))i_+ \\
&= (i_0 - i_+) + \widehat{\phi_2}(i_0)i_+  \\
&= (i_0 - i_+) + \widehat{\phi_2}(i_0)i_+ .
\end{aligned}
\end{equation*}
The third equality holds because\\
\begin{equation*}
\begin{aligned}
\frac{i_0\cdot i_+}{\phi_1(i_0)} = \frac{(i_+ + i_-)\cdot i_+}{\phi_1(i_0)} = \frac{i_+ + i_+}{\phi_1(i_0)} + \frac{i_+ \cdot i_-}{\phi_1(i_+)}  = 2 - \widehat{\phi_2}(i_0) - 1 = 1 - \widehat{\phi_2}(i_0).
\end{aligned}
\end{equation*}
\underline {Step 4}
 \begin{equation*}
\begin{aligned}
 s_{i_- + \widehat{\phi_2}(i_0)i_+}(i_+) &= i_+ - \la i_+ , i_-' + \widehat{\phi_2}(i_0)i_+'  \ra (i_- + \widehat{\phi_2}(i_0)i_+) \\
 &= i_+ -  ( \la  i_+ , i_-' \ra     +  \widehat{\phi_2}(i_0) \la i_+ , i_+'  \ra    ) (i_- + \widehat{\phi_2}(i_0)i_+) \\
 &= i_+ + (1 - \widehat{\phi_2}(i_0)) (i_- + \widehat{\phi_2}(i_0)i_+)\\
 &  = i_0 + \widehat{\phi_2}(i_0)(i_+ - i_-) - \widehat{\phi_2}^2(i_0)i_+  .
 \end{aligned}
\end{equation*}
\underline {Step 5}\\
Applying $s_{i_+}$ to Step 4, we get, 
\begin{equation*}
\begin{aligned}
s_{i_+}s_{i_- + \widehat{\phi_2}(i_0)i_+} &= (i_0 - i_+) + \widehat{\phi_2}(i_0)i_+ + \widehat{\phi_2}(i_0)(-i_+ - i_0 - \widehat{\phi_2}(i_0)i_+ + \widehat{\phi_2}^2(i_0) i_+ \\
&= i_0 - i_+ - \widehat{\phi_2}(i_0) i_0.
\end{aligned}
\end{equation*}
The first equality is because of step 3. Combining all the above steps we get, 
\begin{equation*}
\begin{aligned}
s_{i_+}s_{i_- + \widehat{\phi_2}(i_0)i_+} s_{i_+}(y) &= s_{i_+}s_{i_- + \widehat{\phi_2}(i_0)i_+} (y - \la y , i_+' \ra i_+ ) \\
&= s_{i_+}[y - \la  y , i_-' + \widehat{\phi_2}(i_0)i_+'   \ra (i_- + \widehat{\phi_2}(i_0)i_+) - \la  y , i_+' \ra s_{i_- + \widehat{\phi_2}(i_0)i_+} (i_+)] \\
&= y - \la y , i_+' \ra i_+ - \la y ,   i_-' + \widehat{\phi_2}(i_0)i_+'  \ra i_0  - \la y , i_+' \ra (i_0 - i_+ - \widehat{\phi_2}(i_0) i_0 ) \\
&= y - \la y , i_-' + i_+' \ra i_0 \\
&= y + \la y , i_0 \ra i_0 \\
&= s_{i_0}(y) .
\end{aligned}
\end{equation*}
Third equality is from using step 5. This proves the lemma.
\end{proof}

\begin{example} Note that the above map is not an embedding into the Weyl group $W_I$ as demonstrated by the following example. \end{example}

\begin{figure}[ht]
\begin{minipage}{0.5\linewidth} % Adjust the width as needed
  \centering
  \begin{tikzpicture}
    % Your first TikZ code here
 % First Ellipse
  \node[ellipse, minimum width=1.8cm, minimum height=4.8cm, draw] (ellipse1) at (0,0) {};
  \node[circle, fill, inner sep=1pt] (dot11) at (0, 1) {};
  \node[circle, fill, inner sep=1pt] (dot10) at (0, -1) {};
  
  % Second Ellipse
  \node[ellipse, minimum width=1.8cm, minimum height=4.8cm, draw] (ellipse2) at (4,0) {};
  \node[circle, fill, inner sep=1pt] (dot21) at (4, 1) {};
  \node[circle, fill, inner sep=1pt] (dot20) at (4, -1) {};

  % Connect Dots with Lines and Arrows in the Middle
  \draw[blue, -,postaction={decorate,decoration={markings,mark=at position 0.5 with {\arrow[thick]{stealth};}}}] (dot11) -- (dot21);
  \draw[blue, -,postaction={decorate,decoration={markings,mark=at position 0.5 with {\arrow[thick]{stealth};}}}] (dot10) -- (dot20);
% Curved Line with Arrow in the Middle Top
  \draw[red, -,postaction={decorate,decoration={markings,mark=at position 0.5 with {\arrow[thick]{stealth};}}}] (dot11) to[out=35, in=145] (dot21);
% Curved Line with Arrow in the Middle Bottom
  \draw[blue, -,postaction={decorate,decoration={markings,mark=at position 0.5 with {\arrow[thick]{stealth};}}}] (dot11) to[out=-35, in=-145] (dot21);
  
  % Curved Line with Arrow in the Middle Bottom
  \draw[blue, -,postaction={decorate,decoration={markings,mark=at position 0.5 with {\arrow[thick]{stealth};}}}] (dot10) to [out=-35, in=-145] (dot20);
  % Curved Line with Arrow in the Middle Bottom
  \draw[red, -,postaction={decorate,decoration={markings,mark=at position 0.5 with {\arrow[thick]{stealth};}}}] (dot10) to[out=35, in=145] (dot20);

  \end{tikzpicture}
\end{minipage}%
\begin{minipage}{0.6\linewidth} % Adjust the width as needed
  \centering
  \begin{tikzpicture}
    % Your second TikZ code here
% First Ellipse
  \node[ellipse, minimum width=1.8cm, minimum height=4.8cm, draw] (ellipse1) at (0,0) {};
  \node[circle, fill, inner sep=1pt] (dot11) at (0, 1) {};
  \node[circle, fill, inner sep=1pt] (dot1-1) at (0, -1) {};

     % Loop with Arrow on the Third Dot
  \draw[blue, -,postaction={decorate,decoration={markings,mark=at position 0.5 with {\arrow[thick]{stealth};}}}] (dot11) to[out=210, in=150, looseness=140] (dot11);
  
    % Loop with Arrow on the Third Dot
  \draw[blue, -,postaction={decorate,decoration={markings,mark=at position 0.5 with {\arrow[thick]{stealth};}}}] (dot11) to[out=200, in=160, looseness=110] (dot11);
 
   % Loop with Arrow on the Third Dot
  \draw[blue, -,postaction={decorate,decoration={markings,mark=at position 0.5 with {\arrow[thick]{stealth};}}}] (dot1-1) to[out=210, in=150, looseness=140] (dot1-1);
  
    % Loop with Arrow on the Third Dot
  \draw[blue, -,postaction={decorate,decoration={markings,mark=at position 0.5 with {\arrow[thick]{stealth};}}}] (dot1-1) to[out=200, in=160, looseness=110] (dot1-1);

  \end{tikzpicture}
\end{minipage}%
\caption{ { \it The first graph is before the edge contraction and the second graph is after the edge contraction. The red edges denote the edges along which the contraction takes place. The blue edges in the first graph form loops after the contraction.}} \label{weyl example figure}
\end{figure}
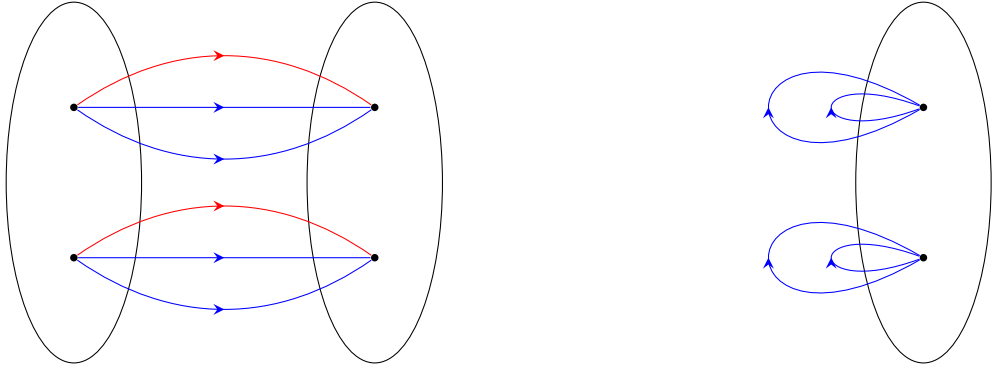

Consider the graph with $I = \{i_+ , i_-\} $ as shown in Figure \ref{weyl example figure} with the first ellipse denoting $[i_+]$ and the second $[i_-]$.
After applying an admissible automorphism $a$ we get a Cartan datum $(I/a, \cdot)$ abbv. as $(I , \cdot)$. Let $Y = \mb{Z}[I]$ and $X = \text{~Hom}(Y , \mb{Z})$ and $\la \cdot, \cdot \ra : Y \times X \longrightarrow \mb{Z}$ be defined as $\la y , x'\ra = x(y)$,  $I \hookrightarrow Y$, $I \hookrightarrow X$ as  $i(j') = \frac{i\cdot j}{\phi_1(i)}~~\in \mb{Z}$. We will show that $s_{i_- + \widehat{\phi_2}(i_0)i_+} \notin W_I$ where $\widehat{\phi_2}(i_0) = 2$. Now,
\begin{equation*}
\begin{aligned}
s_{i_- + 2i_+}(i_+) &= i_+ - \frac{i_+(i_- + 2i_+)}{\phi_1(i_+)} (i_- + 2i_+)\\
& = i_+ - \frac{-6 + 8}{2}(i_- + 2i_+) \\
& = i_+ - (i_- + 2i_+) \\
&= -i_+ - i_-.
\end{aligned}
\end{equation*}

Note that $s_{i_-}(i_+) = i_+ + 3i_-$. Let $s_{i_-}(i_+) = ai_+ + bi_-$ where we assume $b > a > 0$. This implies $s_{i_+}s_{i_-}(i_+) = -ai_+ + b(i_- + 3i_+) = (3b - a)i_+ + bi_-$. Clearly, both $3b - a$ and $b$ are positive. Further, $s_{i_-}s_{i_+}s_{i_-}(i_+) = (3b - a) (i_+ + 3i_-) - bi_- = (3b - a)i_+ + (9b - 3a - b)i_-  = (3b - a)i_+ + (8b - 3a)i_-$. Clearly, $3b - a > 0$, $8b - 3a > 0$ and $8b - 3a > 3b - a$. The last equation indicates that we have an expression $\alpha i_+ + \beta i_-$ where $\beta > \alpha$ and the next step will be to compose with the reflection $s_{i_+}$. Therefore, after repeatedly applying $s_{i_+}, s_{i_-}$ alternatively (starting with $s_{i_-}$), we always get $\alpha, \beta > 0$ and hence never negative. \\

Now consider the case where we begin with $s_{i_+}$. Notice that,\\
\begin{equation*}
\begin{aligned}
s_{i_+}(i_+) &= - i_+, \\
s_{i_-}s_{i_+}(i_+) &= - (i_+ + 3i_-),\\
s_{i_+}s_{i_-}s_{i_+}(i_+) &= - (8i_+ + 3i_-).
\end{aligned}
\end{equation*}

This case is similar to the one considered above, in the sense that we get an expression of the form $-(ai_+ + bi_-)$ such that $ ~ b > a > 0$ (except for the first equation). 
If we show that the `$a$' part strictly increases (in absolute terms) then we are done. Let $s_{i_+}(i_+) = - i_+ =  \gamma$. Then $s_{i_-}(\gamma)$ will be of the form $-(ai_+ + b_-)$ where $b > a > 0$ as before. Further, $s_{i_+}s_{i_-}(\gamma) = - [(3a - b)i_+ + bi_-]$ and note the increase in the value of the coefficient of $i_+$, since $3b - a > a $. 
Now $s_{i_-}s_{i_+}s_{i_-}(\gamma) = -[(3b - a)i_+ + (8b - 3a)i_-]$. Here the coefficient of $i_+$ remains the same. We again end with $8b - 3a > 3b - a > 0$. 
Also, since we have $a = 1$ and $b = 3$ such that $3b - a = 9 -1 = 8 \neq 1$. In the following iteration the coefficient of $i_+$ must be greater than 8. Therefore, the coefficient of $i_+$ in this case can never be $-1$. Hence, $s_{i_- + 2i_+} \notin W_I$.\\

\section{Embedding among Hall Algebras associated to Quivers with Loops}

This section starts with an introduction to representation spaces of quivers with loops. Using Lusztig's induction and restriction diagram, we give a formulation of Hall algebra associated to quivers with loops and prove the embedding among Hall algebras induced by an edge contraction. We end the section by constructing a split short exact sequence of Hall algebras. 

\subsection{Quiver Representation Spaces}
Let $p$ be a prime number and $q = p^e$ for some $e \in \mathbb{N}\backslash \{0\}$. 

Consider a finite oriented graph ($\bfI, \Omega, ', ''$) such that $|\bfI| = n$ and let $a$ be the admissible automorphism as in section (1.2). Let $\mathcal{V}_a$ be a set of $\bfI$-graded $K$-Vector spaces $\bfV = \bigoplus_{\bfi \in \bfI} \bfV_{\bfi}$ where $K$ is the algebraic closure of a finite field $\mathbb{F}_p$ and $a: \bfV \longrightarrow \bfV$ is a linear map such that $a(\bfV)_{\bfi} = \bfV_{a(\bfi)}$ and $a^j|_{\bfV_{\bfi}} = 1|_{\bfV_{\bfi}} ~\forall~ \bfi \in \bfI$ and $\forall~j \geq 1$ such that $a^j(\bfi) = \bfi$. Now to this oriented graph we define a representation space
\begin{eqnarray*} 
\ev = \oplus_{h \in \Omega} \text{Hom}(V_{h^{\prime}}, V_{h^{\prime\prime}}).
\end{eqnarray*}

The algebraic group $\gv = \prod_{\bfi \in \bfI} \text{GL}(\bfV_{\bfi})$ acts on the $\ev$ by conjugation, i.e., for any $g \in \gv$ and $x = (x_h)_{h \in \Omega} \in \ev$~ $(g.x)_h = g_{h^{\prime\prime}}x_hg_{h^{\prime}}^{-1}$. 

Let $\text{Fr}: K \longrightarrow K$ be the Frobenius map defined by raising the elements to a power of $q$. This induces natural Frobenius maps on $\ev$ and $\gv$, denoted by the same notation, by raising the entries of the matrix representation of the above spaces to the power of $q$. This is the same Frobenius map as defined in \cite{Lus98}.

Let $I = \bfI/a$ be the $a$-orbits of $\bfI$ from section (1.2). Let $\nu = \sum_{i\in I}v_i \in \mathbb{N}[I]$. Then $\mathcal{V}_{I, \nu} = \{\bfV \in \mathcal{V}_a ~|~ \text{dim} \bfV_{\bfi} = \nu_i ~\forall \bfi \in i, i \in I = \bfI/a   \}$.

The admissible automorphism $a$ induces an automorphism on $\ev$ and $\gv$ by permuting the edges and the grading on $\bfV$ respectively, i.e. $a: \ev \longrightarrow \ev$ is defined as $a(x_h)_{h \in \Omega} = (x_{a(h)})_{h \in \Omega}$ and $a: \gv \longrightarrow \gv$ given by $a(g_1, g_2, ..., g_n)(v_1, v_2, ..., v_n)$ $= (a(g_1)a(v_1), a(g_2)a(v_2),... , a(g_n)a(v_n))$.

The maps Fr and $a$ are commuting automorphisms on $\ev$ and $\gv$. Let $\bfF = \text{Fr}\circ a$ and call it the twisted Frobenius as seen in \cite{Li23}. Let $\evF$ and $\gvF$ be the fixed points under the action of $\bfF$. Hence $\gvF$ acts on $\evF$ induced by the action of $\gv$ on $\ev$. 

We restrict the further discussion to a graph satisfying the conditions in (\ref{graphassumption}). Let $\whI = I - \{i_+, i_-\} \cup \{i_0\}$. Fix $e \in \Omega$ such that $e^{\prime} \in i_+$ and $e^{\prime\prime} \in i_-$.
Let $\evh = \{ x \in \ev ~|~ x_{h} \text{~is an isomorphism~} \forall h = a^k(e) \text{~for some~} 1 \leq k \leq \phi_1(i_+)   \}$ where $\bfV \in \mathcal{V}_{I, \nu}$ such that dim $V_{i_+}$ $=$ dim $V_{i_-}.$

Clearly, the space $\evh$ is a $\gv$-invariant subvariety of $\ev$. Moreover, we can naturally define a free action of the subgroup $\gv^{[\bfim]} = \prod_{\bfi \in [\bfim]} \text{GL}(\bfV_i)$ of $\gv$ on $\evh$. 

Consider the  quotient map ${\bf q_{V}}: \evh \longleftrightarrow \gv^{[\bfim]} \backslash \evh$. Let $\widehat{\bfV} = \oplus_{\bfi \in \bfI - [\bfim]} \bfV_{\bfi}$. Define a morphism of varieties:\\
$\mu_{\nu}: \evh \longrightarrow \evhat$~,~ $x \mapsto \widehat{x}~~$ by
\begin{equation}
\begin{aligned}
    \widehat{x}_h & = x_h  & \quad &\text{if } h \in \widehat{\Omega} \cap \Omega, \\  
    \widehat{x}_{l_1h_k} & = x_{l_1}x_{h_k} & \quad &\text{if } h_k^{\prime\prime} = l_1^{\prime}, \\  
    \widehat{x}_{\overline{h_k}{l_2}} & = x_{h_k}^{-1}x_{l_2} & \quad &\text{if } h_k^{\prime\prime} = l_2^{\prime\prime}.  
\end{aligned}
\myeqn{} \label{hatmap}
\end{equation}

There is also a natural projection $~~\widehat{}: \gv \longrightarrow {\bf G_{\widehat{V}}}$ by forgetting the $[\bfim]$ components. It is straightforward to check the compatibility relation: $\mu_{\nu}(g.x) = \widehat{g}.\mu_{\nu}(x)~~\forall g \in \gv, ~x \in \evh$. This relation implies that the map $\mu_{\nu}$ factors through the quotient map ${\bf q_{V}}$:
\[
  \begin{tikzcd}[row sep=huge, column sep = huge]
    \evh \arrow{r}{{\bf q_{V}}} \arrow[swap]{dr}{\mu_{\nu}} & \gv^{[\bfim]} \backslash \evh \arrow{d}{\widetilde{\mu}_{\nu}} \\
     & \evhat
  \end{tikzcd}
  \myeqn{}
\]
Following \cite{Li23} (Lemma 2.1.1), we can prove that $\widetilde{\mu}_{\nu}$ is an isomorphism and we hence $\mu_{\nu}$ gets identified with the quotient map ${\bf q_{V}}$.\\
If $\bfV \in \mathcal{V}_{I, \nu}$, then the twisted Frobenius map $\bfF$ restricts to a twisted Frobenius $\bfF$ on $\evh$. Similarly, we can define a twisted Frobenius ${\bf \widehat{F}}$ on $\evhat$ and $\gvhat$. The twisted Frobenius is further compatible with $\mu_{\nu}$, and hence $\mu_{\nu}$ induces a map, again call it $\mu_{\nu}$, from $\evhF$ to $\evhatF$. Due to the compatibility of the projection map $\gvF \longrightarrow \gvhatF$, we have the following commutative diagram:
\[
  \begin{tikzcd}[row sep=huge, column sep = huge]
    \evhF \arrow{r}{{\bf q_{V}}} \arrow[swap]{dr}{\mu_{\nu}} & \gv^{[\bfim], \bfF} \backslash \evhF \arrow{d}{\cong ~\widetilde{\mu}_{\nu}} \\
     & \evhatF
  \end{tikzcd}
  \myeqn{}
\]

Like \cite{Li23}, we will also call $\mu_{\nu}$, a contraction map of $\evF$ along $\{[\bfip], [\bfim]\}$.\\

Having defined $\mu_\nu$, we can also define the inclusion map, denoted by $j_\nu$, from $\evh$ to $\ev$. As it is compatible with the twisted Frobenius, we get a map from $\evhF$ to $\evF$ again denoted the by same notation.

\subsection{Lusztig's Induction and Restriction Diagram}
Let $\bfV$ be an $\bfI$ graded vector space over $K$. We can decompose $\bfV$ as $\bfV = \bfT \oplus \bfW$. Let $\bfS_{\bfW} = \{x \in \ev ~|~ x_h(\bfW_{h^{\prime}}) \subseteq \bfW_{h^{\prime\prime}} \}$, i.e., $W$ is $x$-invariant. Hence we can define $x^{\bfW} \in \ew$ as the restriction of $x \in \bfS_{\bfW}$ to $\bfW$. We can further define $x^{\bfT} \in \ew$ by passage through the quotient $\bfV / \bfW$. Following \cite{Lus10}, we have the following restriction diagram,
\begin{equation}
\ev \xleftarrow{~~~\iota~~~} \bfS_{\bfW} \xrightarrow{~~~\kappa~~~} \et \times \ew
    \myeqn{}
\end{equation}
where $\iota$ is the inclusion map and $\kappa$ is defined by $x \mapsto (x^{\bfT}, x^{\bfW}) $. 

Now to define the induction diagram, we need a stabilizer of $\bfW$ in $\gv$. Call it $Q$. $Q$ is a parabolic subgroup of $\gv$. We denote by $U$ the unipotent radical of $Q$. Now any element of $Q$ can be written as a block matrix:
$\begin{bmatrix}
\begin{array}{c|c}
\text{Morphism from~} T {~to~} T & 0 \\ \hline
\text{Morphism from~} T {~to~} W & \text{Morphism from~} W {~to~} W \\
\end{array}
\end{bmatrix}$\\
This induces a subjective map from $Q$ to $\gt \times \gw$ with kernel $U$. Hence, $Q/U \cong \gt \times \gw$.

Now, we can define the action of $Q$ on the set $\gv \times {\bf S_W}$ by $b.(g, x) = (gb^{-1}, b.x)$ where the action on the second component is the restriction of the action of $\gv$ on $\ev$. $U$ also acts on $\gv \times {\bf S_W}$ via restriction. Let $\bfE'$ and $\bfE''$ be the orbits of the action of $U$ and $Q$ on $\gv \times {\bf S_W}$ respectively. Let $p_1$ be the map from $\bfE'$ to $\et \times \ew$ defined as $[g, x]_U \mapsto (x^\bfT, x^\bfW)$. $p_2$ is the map that takes the $U$-orbit to the corresponding $Q$-orbit. Note that $p_2$ is a $Q/U (\cong \gt \times \gw)$ principal bundle. Further, $p_3$ maps an element of the $Q$-orbit, $[g, x]_Q$, to $g.x$. Combining these maps gives rise to the induction diagram shown below: 
\begin{equation}
    \et \times \ew \xleftarrow{~~~p_1~~~} \bfE' \xrightarrow{~~~p_2~~~} \bfE'' \xrightarrow{~~~p_3~~~} \ev
    \myeqn{}
\end{equation}
Note the the restriction and induction diagrams are compatible with the twisted Frobenius $\bfF$ provided $\bfV, \bfT, \bfW$ are in $\mathcal{V}_a$. So we have the following twisted induction and restriction diagrams with the maps being restriction of the corresponding maps in the restriction and induction diagram:
\begin{equation}
\evF \xleftarrow{~~~\iota~~~} \bfS^{\bfF}_{\bfW^\bfF} \xrightarrow{~~~\kappa~~~} \etF \times \ewF
    \myeqn{}
\end{equation}
\begin{equation}
    \etF \times \ewF \xleftarrow{~~~p_1~~~} \bfE'^{\bfF} \xrightarrow{~~~p_2~~~} \bfE''^{\bfF} \xrightarrow{~~~p_3~~~} \evF
    \myeqn{} \label{induction}
\end{equation}

Suppose $\nu = \tau + \omega$ where $\nu = \text{dim~} \bfV$, $\tau = \text{dim~} \bfT$ and $\omega = \text{dim~} \bfT$, each compatible with the grading. Let $(\bfV, \bfT, \bfW)$ be the vector spaces in $\mathcal{V}_{I, \nu} \times \mathcal{V}_{I, \tau} \times \mathcal{V}_{I, \omega}$. Fix this triple. 
We can extend the restriction diagram by introducing the $\heart$ and edge contracted (or ~$\widehat{}$~~) versions of $\iota$ and $\kappa$, denoted by $\iota^{\heart}$, $\kappa^{\heart}$,  $\widehat{\iota}$ and $\widehat{\kappa}$ respectively. Here $\iota^{\heart}$ and $\kappa^{\heart}$ are just the restrictions of $\iota$ and $\kappa$ respectively. 
Following, \cite{Li23}, we have the extended restriction diagram:
\[
\begin{tikzcd}[row sep=huge, column sep = huge]
\evF  &  {\bfS_{\bfW}^{\bfF}}  \arrow[r,"\kappa"] \arrow[l, swap, "\iota"] &
\etF \times \ewF
\\
\evhF \arrow[u,"j_\nu"] \arrow[d, swap, "\mu_\nu"] &  {\bfS_{\bfW}^{\heart,\bfF}} \arrow[u,"j_\nu'"] \arrow[d, swap, "\mu_\nu'"] \arrow[r,"\kappa^{\heart}"] \arrow[l, swap, "\iota^{\heart}"] &
\ethF \times \ewhF \arrow[u, swap, "j_\tau \times j_\omega"] \arrow[d,  "\mu_\tau \times \mu_\omega"]
\\
\evhatF  &  {\bfS_{\widehat{\bfW}}^{\widehat{\bfF}}}  \arrow[r,"\widehat{\kappa}"] \arrow[l, swap, "\widehat{\iota}"] &
\ethatF \times \ewhatF
\end{tikzcd}
\myeqn{} \label{extrestriction}\]
The vertical maps $j_{\nu}'$ and $\mu_{\nu}'$ are the restriction maps of $j_\nu$ and $\mu_\nu$ respectively.\\

We make the following observations about the extended restriction diagram.

\begin{lemma}
All the squares in the diagram (\ref{extrestriction}) are commutative. Further, the top two squares are Cartesian.  \end{lemma} 
\begin{proof}
 It is clear that the top two squares and the bottom left square are commutative, since we are dealing with restrictions and inclusions. For the bottom right square, since restriction to $\bfW$ and contraction map are commutative, we have the commutativity.   
 
Let us prove that the top left square is Cartesian. Let $C$ be the fibered product of $\bfS_\bfW^\bfF$ and $\evhF$. Consider the following diagram where $\pi_1$ and $\pi_2$ are the projections onto the first and second components respectively. \\
\[\begin{tikzcd}
\evF & & \bfS_\bfW^\bfF \arrow[ll, "\iota", swap] \\
& C' \arrow[ru, "\pi_1"] \arrow[dl, "\pi_2", swap] \arrow[dr, bend left]
&  \\
\evhF \arrow[uu, "j_\nu"] &  & \bfS_\bfW^{\heart, \bfF} \arrow[ll, "\iota^{\heart}"] \arrow[uu, "j_\nu'"] \arrow[ul, dotted, bend left]
\myeqn{}
\end{tikzcd} \]
Since $C'$ is the fibered product, we have $\iota\pi_1(x, \overline{x}) = j_\nu\pi_2(x, \overline{x})$ where $x \in \bfS_\bfW^\bfF$ and $\overline{x} \in \evhF$. This implies $x = \overline{x}$ and hence $x \in \bfS_\bfW^\bfF \cap \evhF = \bfS_\bfW^{\heart, \bfF}$. This makes the top left Cartesian.

Next we prove that the top right square is Cartesian. Again, we let $C'' = \bfS_\bfW^\bfF \times (\ethF \times \ewhF)$, a fibered product. Similarly to the above, consider the following diagram:\\
\[\begin{tikzcd}
\bfS_\bfW^\bfF \arrow[rr, "\kappa"] & & \etF \times \ewF \\
& C'' \arrow[lu, "\pi_1"] \arrow[dr, "\pi_2", swap] \arrow[dl, bend left]
&  \\
\bfS_\bfW^{\heart, \bfF} \arrow[ur, dotted, bend left] \arrow[uu, "j_\nu'"] \arrow[rr, "\kappa^{\heart}"] &  & \ethF \times \ewhF  \arrow[uu, "j_\tau \times j_\omega", swap] 
\myeqn{}
\end{tikzcd} \]
Let $x \in \bfS_\bfW^\bfF$ and $(y^\bfT, y^\bfW) \in \ethF \times \ewhF$. Then we have, $\kappa(x) = j_\tau \times j_\omega(y^\bfT, y^\bfW)$. Hence, $(x^\bfT, x^\bfW) = (y^\bfT, y^\bfW)$, i.e., $x^\bfT \in \ethF,~  x^\bfW \in \ewhF$. Now using that fact that $x_h$ is an isomorphism iff $x_h^{\bfT}$ and $x_h^\bfW$ are isomorphisms, we get $x \in \bfS_\bfW^{\heart, \bfF}$. The lemma is proved. 
\end{proof}

\begin{lemma}
\[\begin{tikzcd}
{\bfS_{\bfW}^{\heart,\bfF}} \arrow[dr, dotted, "\tilde{\kappa}"] \arrow[dd, "\mu_\nu'", swap] \arrow[rr, "\kappa^{\heart}"] & & \ethF \times \ewhF \arrow[dd, "\mu_\tau \times \mu_\omega"]  \\
& D \arrow[ld, "\pi_1"] \arrow[ur, "\pi_2", swap] 
&  \\
{\bfS_{\widehat{\bfW}}^{\widehat{\bfF}}}  \arrow[rr, "\widehat{\kappa}"] &  & \ethatF \times \ewhatF  
\myeqn{}
\end{tikzcd} \]
Consider the bottom right square of the extended restriction diagram. Let $D$ be the fibered product of ${\bfS_{\widehat{\bfW}}^{\widehat{\bfF}}} $ and $\ethF \times \ewhF$. Then the unique map $\widetilde{\kappa}: {\bfS_{\bfW}^{\heart,\bfF}} \longrightarrow D$ is a vector bundle whose fiber is isomorphic to $(\mathbb{F}_{q^{\phi_(i_+)}})^{\tau_{i_+}.\omega_{i_-} }$.
\end{lemma}
\begin{proof}
 The proof is similar to the proof in \cite{Li23} (Lemma 2.2.2), however, the proof is provided for the reader's convenience and completeness. Let $(\widehat{x}, x^T, x^W) \in D ~(= {\bfS_{\widehat{\bfW}}^{\widehat{\bfF}}} \times (\ethF \times \ewhF) )$.  By definition of the fibered product, we have $\widehat{\kappa} \circ \pi_1 (\widehat{x}, x^T, x^W) = \mu_\tau \times \mu_\omega \circ \pi_2 (\widehat{x}, x^T, x^W) \implies \widehat{\kappa}(\widehat{x}) = \mu_\tau \times \mu_\omega(x^T, x^W)$. In order to determine the fiber of such a triple, we need to determine the $(x_h)_{h \in \Omega}$-components of a $x \in {\bfS_{\bfW}^{\heart,\bfF}}$ such that $\widetilde{\kappa}(x)$ is equal to that triple. From (\ref{hatmap}), $\widehat{x}_h = x_h$ if $h \in \widehat{\Omega} \cap \Omega$, results in the determination of all components of the form $x_h$ where $h \in \widehat{\Omega} \cap \Omega$. We are left to determine the components that are of the form $(x_{h_k}, x_{l_1}, x_{l_2})$ where $h_k'' = l_1'$ and $h_k'' = l_2''$. Since, $\bfV_h = \bfT_h \oplus \bfW_h$ for every edge $h$, we can see that,\\
\[ x_{h_k} = \left[ \begin{array}{cc} 
x_{h_k}^T &  0\\
x_{h_k}^{T \rightarrow W} & x_{h_k}^W 
\end{array} \right] ~~ 
x_{l_1} = \left[ \begin{array}{cc} 
x_{l_1}^T &  0\\
x_{l_1}^{T \rightarrow W} & x_{l_1}^W 
\end{array} \right]  ~~ 
x_{l_2} = \left[ \begin{array}{cc} 
x_{l_2}^T &  0\\
x_{l_2}^{T \rightarrow W} & x_{h_k}^W 
\end{array} \right]\]
We notice that the only thing here left to be determined is $x_{h_k}^{T \rightarrow W}, x_{l_1}^{T \rightarrow W}$ and $x_{l_2}^{T \rightarrow W}$. Using the second and the third condition in (\ref{hatmap}) yields, 
\[ \widehat{x}_{l_1h_k}  = x_{l_1}x_{h_k} \implies \left[ \begin{array}{cc} 
\widehat{x}_{l_1h_k}^T &  0\\
\widehat{x}_{l_1h_k}^{T \rightarrow W} & \widehat{x}_{l_1h_k}^W 
\end{array} \right] = \left[ \begin{array}{cc} 
x_{l_1}^T &  0\\
x_{l_1}^{T \rightarrow W} & x_{l_1}^W 
\end{array} \right] ~ \left[ \begin{array}{cc} 
x_{h_k}^T &  0\\
x_{h_k}^{T \rightarrow W} & x_{h_k}^W 
\end{array} \right]
\] 

\[\widehat{x}_{\overline{h_k}{l_2}} = x_{h_k}^{-1}x_{l_2}  \implies \widehat{x}_{l_1h_k}  = x_{l_1}x_{h_k} \implies \left[ \begin{array}{cc} 
\widehat{x}_{h_kl_2}^T &  0\\
\widehat{x}_{h_kl_2}^{T \rightarrow W} & \widehat{x}_{h_kl_2}^W 
\end{array} \right] = \left[ \begin{array}{cc} 
x_{h_k}^T &  0\\
x_{h_k}^{T \rightarrow W} & x_{h_k}^W 
\end{array} \right]^{-1} ~  \left[ \begin{array}{cc} 
x_{l_2}^T &  0\\
x_{l_2}^{T \rightarrow W} & x_{l_2}^W 
\end{array} \right]\]

These lead to the following equations:
\[ \widehat{x}_{l_1h_k}^{T \rightarrow W} = x_{l_1}^{T \rightarrow W}x_{h_k}^T + x_{l_1}^W x_{h_k}^{T \rightarrow W} \text{~~if~}  
h_k'' = l_1' \]
\[ \widehat{x}_{h_kl_2}^{T \rightarrow W} = x_{l_2}^T(-(x_{h_k}^W)^{-1}x_{h_k}^{T \rightarrow W} (x_{h_k}^{T})^{-1} ) + x_{l_2}^{T \rightarrow W}(x_{h_k}^W)^{-1} \text{~~if~}  
h_k'' = l_2''  \]
Finally, \[ x_{l_1}^{T \rightarrow W}  = (\widehat{x}_{l_1h_k}^{T \rightarrow W} - x_{l_1}^W x_{h_k}^{T \rightarrow W}) (x_{h_k}^T)^{-1} \text{~~if~}  
h_k'' = l_1' \]
\[ x_{l_2}^{T \rightarrow W} =  [ \widehat{x}_{h_kl_2}^{T \rightarrow W} + x_{l_2}^T((x_{h_k}^W)^{-1}x_{h_k}^{T \rightarrow W} (x_{h_k}^{T})^{-1} )] (x_{h_k}^W) \text{~~if~}  
h_k'' = l_2''  \]

This shows that $x_{l_1}^{T \rightarrow W}$ and $x_{l_2}^{T \rightarrow W}$ are dependent on $x_{h_k}^{T \rightarrow W}$. Hence, the fiber is completely determined by the element $x_{h_k}^{T \rightarrow W}$. Notice that $x_{h_k}^{T \rightarrow W}: \bfT_{h_k'} \longrightarrow \bfW_{h_k''}$ implying it is a $\omega_{i_{-}} \times \tau_{i_+}$ matrix, whose entries are in $(\mathbb{F}_{q^{\phi_(i_+)}})$. Therefore, the fiber is isomorphic to $(\mathbb{F}_{q^{\phi_(i_+)}})^{\tau_{i_+}.\omega_{i_-} }$.
\end{proof}

Similar to the extension of the restriction diagram, we can extend the induction diagram as shown in \cite{Li23} and note down the observations in the lemma below the diagram.
\[
\begin{tikzcd}[row sep=huge, column sep = huge]
\etF \times \ewF  &       \bfE'^{\bfF}   \arrow[r,"p_2"] \arrow[l, swap, "p_1"] &
\bfE''^{\bfF} \arrow[r, "p_3"] & \evF
\\
\ethF \times \ewhF \arrow[u,"j_\tau \times j_\omega"] \arrow[d, swap, "\mu_\tau \times \mu_\omega"] &   \bfE'^{\heart, \bfF} \arrow[u,"j_\nu'"] \arrow[d, swap, "\mu_\nu'"] \arrow[r,"p_2^{\heart}"] \arrow[l, swap, "p_1^{\heart}"] &  \bfE''^{\heart\bfF} \arrow[u, swap, "j_\nu''"] \arrow[d,  "\mu_\nu''"] \arrow[r, "p_3^{\heart}"] & \evhF \arrow[u,"j_\nu", swap] \arrow[d, "\mu_\nu"]
\\
\ethatF \times \ewhatF  &       \widehat{\bfE}'^{\widehat{\bfF}}   \arrow[r,"\widehat{p_2}"] \arrow[l, swap, "\widehat{p_1}"] &
\widehat{\bfE}''^{\widehat{\bfF}} \arrow[r, "\widehat{p_3}"] & \evhatF
\end{tikzcd}
\myeqn{} \label{extinduction} \]

Here, the first row is (\ref{induction}), the second row is just the restriction of the maps on the first row and the third row represents the edge contracted versions of the first row. Also, the maps $j_{\nu}'$ and $j_\nu''$ are inclusion maps and $\mu_\nu'$ and $\mu_\nu''$ are the maps defined by applying $\mu_\nu$ and the natural projection~~($~\widehat{~}~$) to its components.   

\begin{lemma}
All the squares in the extended induction diagram (\ref{extinduction}) are commutative. Further, all the top squares and the bottom right square is Cartesian. 
\end{lemma} 
\begin{proof}
Commutativity follows from the definition. The proof of top left square being Cartesian follows closely to the proof of the top right square of the extended restriction diagram (2.2.e) being Cartesian.

To prove that the top middle square is Cartesian, we let $C = \bfE'^{\bfF} \times \bfE''^{\heart\bfF}$, a fibered product. Consider the following diagram:
\[\begin{tikzcd}
\bfE'^{\bfF} \arrow[rr, "p_2"] & & \bfE''^{\bfF} \\
& C \arrow[lu, "\pi_1"] \arrow[dr, "\pi_2", swap] \arrow[dl, bend left]
&  \\
\bfE'^{\heart\bfF} \arrow[ur, dotted, bend left] \arrow[uu, "j_\nu'"] \arrow[rr, "p_2^{\heart}"] &  & \bfE''^{\heart\bfF}  \arrow[uu, "j_\nu''", swap] 
\myeqn{}
\end{tikzcd} \]
By the definition of a fibered product, we have $p_2([g, x]_U) = j_\nu''([\overline{g}, \overline{x}]_Q)$ where $[g, x]_U \in \bfE'^{\bfF}$ and $[\overline{g}, \overline{x}]_Q \in \bfE''^{\heart\bfF}$. This implies, $[g, x]_Q = [\overline{g}, \overline{x}]_Q$. Hence, there exists a $r \in Q$ such that $(g, x) =  r. (\overline{g}, \overline{x}) = (\overline{g} r^{-1}, r. \overline{x})$. Therefore,  $x = r.\overline{x}$, implying $x \in \bfS_\bfW^{\heart, \bfF}$ and we finally have $[g, x]_U \in \bfE'^{\heart, \bfF}$. In a similar fashion, the top right square can be proved Cartesian. 

Next, we proceed to prove that the bottom right square is Cartesian. To prove this, we will use the following:
$\bfE''^{\bfF} \cong \{ (x, \bfU) ~|~ x \in \evF, \bfV/\bfU \cong \bfT, \bfU \cong \bfW\}$ which leads to 
$\bfE'^{\bfF} \cong \{ (x, \bfU, g^{\prime}, g'' ) ~|~ (x, \bfU) \in \bfE''^{\bfF},~g': \bfV/\bfU \cong \bfT, g'':\bfU \cong \bfW\}$. Similarly, we have $\widehat{\bfE}''^{\widehat{\bfF}} \cong \{ (\widehat{x}, \widehat{\bfU}) ~|~ \widehat{x} \in \evhatF, \widehat{\bfV}/\widehat{\bfU} \cong \widehat{\bfT},\widehat{\bfU} \cong \widehat{\bfW}\}$ and $\bfE''^{\heart, \bfF} \cong \{ (x, \bfU) ~|~ x \in \evhF, \bfV/\bfU \cong \bfT,\bfU \cong \bfW\}$. 
Let $C' = \evhF \times \widehat{\bfE}''^{\widehat{\bfF}}$.
\[\begin{tikzcd}
\bfE''^{\heart, \bfF} \arrow[dr, dotted, bend left] \arrow[dd, "\mu_\nu''", swap] \arrow[rr, "p_3^{\heart}"] & & \evhF \arrow[dd, "\mu_\nu"]  \\
& C' \arrow[ld, "\pi_1"] \arrow[ur, "\pi_2", swap] \arrow[ul, bend left]
&  \\
\widehat{\bfE}''^{\widehat{\bfF}}  \arrow[rr, "\widehat{p_3}"] &  & \evhatF  
\myeqn{}
\end{tikzcd} \]

Define $\bfU_\bfi = \widehat{\bfU}_\bfi, ~ \bfi \neq \bfim$ and $\bfU_{\bfim} = x_{l_2}(\bfU_{l_2'})$ where the edge $l_2$ is from (\ref{hatmap}). Now, $x_{h_k}(\bfU_{h_k'}) = x_{h_k}(x_{h_k}^{-1}x_{l_2}(\widehat{\bfU}_{l_2'})) = (x_{h_k}x_{h_k}^{-1})(x_{l_2}(\widehat{\bfU}_{l_2'})) \subseteq x_{l_2}(\widehat{\bfU}_{l_2'}) = x_{l_2}(\bfU_{l_2'}) = \bfU_\bfim$. Therefore, we get a map from $C' \to \bfE''^{\heart, \bfF}$ defined by $(x, (\widehat{x}, \widehat{\bfU})) \mapsto (x, \bfU)$ and the result follows.\end{proof}

\begin{lemma}
Let $C$ be the fibered product of $\bfE''^{\heart\bfF}$ and $\widehat{\bfE}'^{\widehat{\bfF}}$. Then there exists maps \( p' : \bfE'^{\heart\bfF} \to C \),~ $\widehat{p}: C \to \bfE''^{\heart, \bfF}$ and $\widehat{\mu}: C \to \widehat{\bfE}'^{\widehat{\bfF}}$ such that the middle square in the extended induction diagram factors through $C$ and makes the following diagram commute:
\[
\begin{tikzcd}[row sep=large, column sep = large]
\bfE'^{\heart, \bfF} \arrow[dr, dotted, "{p'}" description] \arrow[dd, swap, "\mu_\nu'"] \arrow[rr,"p_2^{\heart}"] & & \bfE''^{\heart, \bfF} \arrow[dd,  "\mu_\nu''"]  
\\
& C \arrow[ur, "\widehat{p}"] \arrow[dl, "\widehat{\mu}", swap] & \\
\widehat{\bfE}'^{\widehat{\bfF}}   \arrow[rr,"\widehat{p_2}"]  & &
\widehat{\bfE}''^{\widehat{\bfF}}
\myeqn{}
\end{tikzcd}
 \]\\ 
Moreover, the map \( p' \) defines a \({\bf G}_{\bfT}^{[\bfim], \bfF} \times {\bf G}_{\bfW}^{[\bfim], \bfF}\)-bundle.
\end{lemma}

\begin{proof}
The first part of the proof follows from the universal property of fibered products. \\
When defining the induction diagram, we noted that \( p_2 \) is a $Q^{\bfF}/U^{\bfF} (\cong \gtF \times \gwF)$-bundle. It follows that \( p_2^{\heartsuit} \) is a $Q^{\bfF}/U^{\bfF} (\cong \gtF \times \gwF)$-bundle. This implies that $\bfE'^{\heartsuit, \bfF}$ has a natural fiber structure compatible with the group actions of \( \gtF \) and \( \gwF \). Note that the map, \( \widehat{p}: C \to \bfE''^{\heart, \bfF}\) defined as $\widehat{p}([g, x]_Q,[h, y]_{\widehat{U}}) = [g, x]_Q$ has fibers in $\widehat{\bfE}'^{\widehat{\bfF}}$, implying it is a  $\widehat{Q}^{\widehat{\bfF}}/\widehat{U}^{\widehat{\bfF}} (\cong \gthatF \times \gwhatF)$-bundle. Therefore, by the properties of fiber bundles, \( p' \) inherits the same bundle structure as a \({\bf G}_{\bfT}^{[\bfim], \bfF} \times {\bf G}_{\bfW}^{[\bfim], \bfF}\)-bundle and the lemma is proved.
\end{proof}

\subsection{Embedding among Hall Algebras}
Fix a square root $q^{1/2}$ of $q$ in the field $\mb C$ of complex numbers and let $\mb Q(q^{1/2})$ be the subfield of $\mb C$
generated by $\mb Q$ and $q^{1/2}$.

Let $\phi: X\to Y$ be a map between two finite sets. If $f: Y\to \mb Q(q^{1/2})$ is a function, we define the pullback of $f$ along $\phi$ by
$\phi^* (f) = f \phi$. If $f' : X\to \mb Q(q^{1/2})$ is a function, we define  
a push-forward function $\phi_!(f') : Y\to \mb Q(q^{1/2})$ by $\phi_!(f')(y) = \sum_{x\in \phi^{-1}(y)} f'(x)$ for all $y\in Y$.

Fix $\bfV \in \mathcal V_{I,\nu}$ where $\nu \in \mb{N}[I]$.
To each $\nu \in \mb{N}[I]$ we define a vector space associated with it as follows:
\begin{equation}
    \hvO = \{f: \evF \to \mb{Q}(q^{1/2}) ~|~ f(g.x) = f(x) ~\forall g \in \gvF, x \in \evF \}.
    \myeqn{}
\end{equation}
i.e., $f$ is a $\gvF$-invariant. Let $\nu, \tau, \omega \in \mb{N}[I]$ such that $\nu = \tau + \omega$. Further, we let $(\bfV, \bfW, \bfT) \in \mathcal{V}_{I, \nu} \times \mathcal{V}_{I, \tau} \times \mathcal{V}_{I, \omega}$. 
Now, we define the Hall algebra $\hO$ by taking $\hvO$ and  summing over all the dimension vectors $\nu$, i.e.,
\begin{equation}
    \hO = \bigoplus_{\nu \in \mb{N}[I]} \hvO.
    \myeqn{}
\end{equation}
and defining multiplication as follows:
\[
\begin{aligned}
    \text{Ind}^{\nu}_{\tau, \omega}&: \htO \otimes \hwO \longrightarrow \hvO; \\
    \text{Ind}^{\nu}_{\tau, \omega} &= (q^{1/2})^{-m_{\Omega}(\tau, \omega)} \ds\frac{1}{\# \gtF \times \gwF}(p_3)_!(p_2)_!(p_1)^*
\end{aligned} \myeqn{} \]
where $m_\Omega(\tau, \omega) = \ds\sum_{\bfi \in \bfI} \tau_{\bfi}\omega_{\bfi} + \ds\sum_{h \in \Omega}\tau_{h'}\omega_{h''} $ and the maps $p_1, p_2$ and $p_3$ are from the induction diagram.

In \cite{Lus98}, Lusztig proved that the above-mentioned Hall algebra is a unital associative algebra over $\mb{Q}(q^{1/2})$. Let $f_1 \in \htO$ and $f_2 \in \hwO$. We then write,
\[
\begin{aligned}
   f_1 \star f_2 &=  \ds\frac{1}{\# \gtF \times \gwF}(p_3)_!(p_2)_!(p_1)^*(f_1 \otimes f_2)   \text{~~~~and} \\
   f_1 \circ f_2 &= (q^{1/2})^{-m_{\Omega}(\tau, \omega)} f_1 \star f_2 =   \text{Ind}^{\nu}_{\tau, \omega}(f_1 \otimes f_2).
\end{aligned}
\myeqn{} \]

Similarly, the restriction diagram helps to define the map dual to $ \text{Ind}^{\nu}_{\tau, \omega}$ as follows:
\[
\begin{aligned}
    \text{Res}^{\nu}_{\tau, \omega}&: \hvO \longrightarrow \htO \otimes \hwO  \\
    \text{Res}^{\nu}_{\tau, \omega} &= (q^{1/2})^{-m^{*}_{\Omega}(\tau, \omega)} \kappa_!\iota^*
\end{aligned} \myeqn{} \]
where $m^*_\Omega(\tau, \omega) = -\ds\sum_{\bfi \in \bfI} \tau_{\bfi}\omega_{\bfi} + \ds\sum_{h \in \Omega}\tau_{h'}\omega_{h''} $ and the maps $\kappa$ and $\iota$ are from the restriction diagram.

Let $r_{\Omega} = \ds\bigoplus_{\nu = \tau + \omega}  \text{Res}^{\nu}_{\tau, \omega}: \hO \to \hO \otimes \hO $. 
In \cite{Lus98}, we see that $(\hO, r_{\Omega})$ forms an associative algebra by defining multiplication in $\hO \otimes \hO$ as $(f_1 \otimes f_2)(g_1 \otimes g_2) = (q^{1/2})^{\nu_2\cdot\mu_1} f_1 \circ g_1 \otimes f_2 \circ g_2$ for all $f_1, f_2, g_1, g_2 \in \hO$ such that $(f_1, f_2) \in {\bf H}_{\nu_1} \times {\bf H}_{\nu_2}$ and $(g_1, g_2) \in {\bf H}_{\mu_1} \times {\bf H}_{\mu_2}$. \\
Having defined $\hvO$, we define the following:
\[
\begin{aligned}
    \hvOh &= \{f: \evhF \to \mb{Q}(q^{1/2}) ~|~ f(g.x) = f(x) ~\forall g \in \gvF, x \in \evhF \},\\
    \hvOhat &= \{f: \evhatF \to \mb{Q}(q^{1/2}) ~|~ f(g.x) = f(x) ~\forall g \in \gvhatF, x \in \evhatF \},
\end{aligned}
\]
where $\nu \in \mathbb{N}[\widehat{I}]$. Just like in \cite{Li23}, this gives 2 new Hall algebras, 
\[
\hOh = \bigoplus_{\nu \in \mb{N}[\widehat{I}]} \hvOh \text{~~and~~} \hOhat = \bigoplus_{\nu \in \mb{N}[\widehat{I}]} \hvOhat,
\]
where multiplication is defined using the second and the third row in the extended induction diagram (\ref{extinduction}) respectively.

The inclusion map $j_\nu: \evhF \to \evF$, induces the map,
$j_{\nu!}: \hvOh \to \hvO$ by extension by zero, i.e., 
$j_{\nu!}(f)(x) = \left\{ \begin{array} {cc} f(x) \ \ \ \ \ \ \ &\text{if} \ \ x \in \evhF\\
0 \ \ \ \ &\text{if} \ \  x \in \evF \backslash \evhF \end{array} \right. $.

Taking sum over $\nu \in \mb{N}[\widehat{I}]$, we get $j_!: \hOh \to \hO$, and it can be shown that it is an algebra homomorphism over $\mb{Q}(q^{1/2})$, by using the fact that the top square in the extended induction diagram is Cartesian.\\
We may also define the pullback of the contraction map $\mu_\nu: \evhF \to \evhatF$ by $\mu_\nu^*: \hvOhat \to \hvOh$, $f \mapsto \mu_\nu \circ f$. Note that this map is an isomorphism of vector spaces. We twist this map and define \[\mu_\nu^\star = (q^{-1/2})^{\nu_{i_-}^2 \phi_1(i_-)}\mu_\nu^*. \]
Again, taking sum over $\nu \in \mb{N}[\widehat{I}]$, we get \[\mu^\star: \hOhat \to \hOh \] 

We will later prove that this is an algebra homomorphism, and in fact an isomorphism. 
For now, assuming that $\mu^{\star}$ is an isomorphism, we define, 
\begin{equation*}
\begin{aligned}
\mu_{\star} = (\mu^{\star})^{-1} = \ds\frac{(q^{-1/2})^{\nu_{i_-}^2 \phi_1(i_-)}}{\# \text{GL}(\bfV_{i_-})} \mu_!.
\end{aligned}
\end{equation*}

The next theorem gives the main result of this paper. The theorem confers a embedding among the Hall algebras $\hOhat$ and $\hO$. The proof of the theorem is similar to the proof of theorem 2.3.1 in \cite{Li23} with some changes accounting for additional edges between the orbits, about which edges are contracted, and formation of loops after the edge contraction. Nonetheless, the entire proof is written here for convenience. 

\begin{theorem}
The composition of maps defined by $\psi = j_!\circ \mu^\star: \hOhat \to \hO$ is an embedding of Hall algebras. \label{main embedding}
\end{theorem}
\begin{proof} Above we have proved that $j_!: \hOh \to \hO$, is an algebra homomorphism. It remains to show that $\mu^{\star}$ is also a homomorphism. To do that, we need to show the multiplicative compatibility of $\mu^{\star}$, i.e., $\mu_\nu^\star \widehat{\text{Ind}}_{\tau, \omega}^{\nu} = \text{Ind}_{\tau, \omega}^{\nu. \heart}(\mu_{\tau}^\star \otimes \mu_\omega^\star)$. From the extended induction diagram (\ref{extinduction}) we have, 
\[ \begin{aligned}
\mu^*_\nu (\widehat p_3)_! (\widehat p_2)_! \widehat p_1^* 
& =  (p^\heart_3)_! \mu''^*_\nu (\widehat p_2)_! \widehat p^*_1\\
& = (p^\heart_3)_! \widehat p_! \widehat \mu^* \widehat p^*_1\\
& = \frac{1}{\# \bfG^{\bfim, \bfF}_{\bfT} \times \bfG^{\bfim,\bfF}_{\bfW} } (p^\heart_3)_! (p^\heart_2)_! \mu'^*_{\nu} \widehat p^*_1\\
&=  \frac{1}{\# \bfG^{\bfim, \bfF}_{\bfT} \times \bfG^{\bfim,\bfF}_{\bfW} } (p^\heart_3)_! (p^\heart_2)_!  (p^\heart_1)^* (\mu^*_\tau \otimes \mu^*_{\omega}) ,
\end{aligned} \myeqn{} \label{embedcalc}
\] 

where the third equality is due to Lemma (2.2.4),
\[
(p^\heart_2)_! \mu'^*_\nu = 
\widehat p_! p'_! p'^* \widehat \mu^* = \# (\bfG^{\bfim, \bfF}_{\bfT} \times \bfG^{\bfim,\bfF}_{\bfW}) \widehat p_! \widehat \mu^*. \]
By (\ref{embedcalc}), we have,
\[
\begin{aligned}
\mu^\star_{\nu} \widehat{\text{Ind}}^{ \nu}_{\tau,\omega}
& = (q^{1/2})^{-\nu^2_{i_-} \phi_1(i_-) - m_{\widehat \Omega}( \tau, \omega)} 
\frac{1}{\# \bfG^{\widehat \bfF}_{\widehat \bfT}\times \bfG^{\widehat \bfF}_{\widehat \bfW} } 
\mu^*_\nu (\widehat p_3)_!(\widehat p_2)_! (\widehat p_1)^*\\
&=(q^{1/2})^{-\nu^2_{i_-} \phi_1(i_-) - m_{\widehat \Omega}(\tau,\omega)} 
\frac{1}{\# \bfG^{\bfF}_\bfT\times \bfG^{\bfF}_\bfW} (p^\heart_3)_! (p^\heart_2)_!  (p^\heart_1)^* (\mu^*_\tau \otimes \mu^*_{\omega}) \\
& = (q^{1/2})^{-N} \text{Ind}^{\nu,\heartsuit}_{\tau,\omega} (\mu^\star_{\tau} \otimes \mu^{\star}_{\omega}),
\end{aligned} 
\]

where 
\[
\begin{aligned}
N & = (\nu^2_{i_-} -\tau^2_{i_-} - \omega^2_{i_-}) \phi_1(i_-)  + m_{\widehat \Omega}(\tau,\omega) - m_\Omega(\tau,\omega)\\
&=(\nu^2_{i_-} -\tau^2_{i_-} - \omega^2_{i_-}) \phi_1(i_-)  - 2 \tau_{i_-}\omega_{i_-} (\phi_1(i_-)) \\
&=0.
\end{aligned} \]
The theorem is proved.
\end{proof}

The calculation $m_{\widehat \Omega}(\tau,\omega) - m_\Omega(\tau,\omega) = -2\tau_{i_-}\omega_{i_-} (\phi_1(i_-))$ can be illustrated as follows:

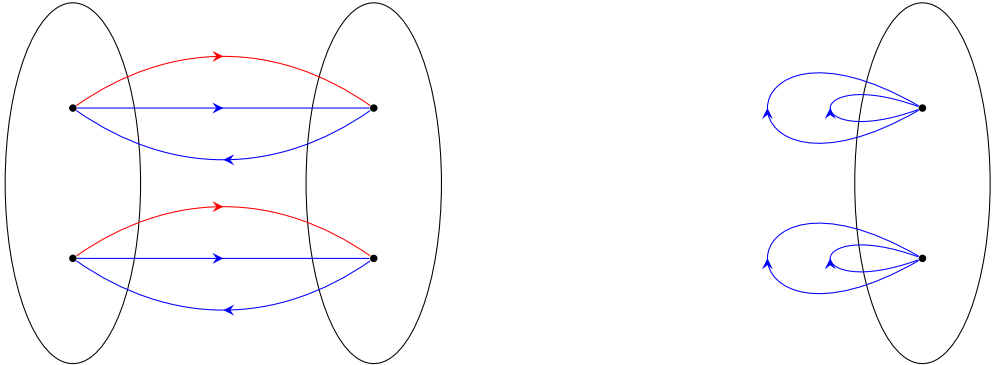
\begin{figure}[ht]
\begin{minipage}{0.5\linewidth} % Adjust the width as needed
  \centering
  \begin{tikzpicture}
    % Your first TikZ code here
 % First Ellipse
  \node[ellipse, minimum width=1.8cm, minimum height=4.8cm, draw] (ellipse1) at (0,0) {};
  \node[circle, fill, inner sep=1pt] (dot11) at (0, 1) {};
  \node[circle, fill, inner sep=1pt] (dot10) at (0, -1) {};
  
  % Second Ellipse
  \node[ellipse, minimum width=1.8cm, minimum height=4.8cm, draw] (ellipse2) at (4,0) {};
  \node[circle, fill, inner sep=1pt] (dot21) at (4, 1) {};
  \node[circle, fill, inner sep=1pt] (dot20) at (4, -1) {};

  % Connect Dots with Lines and Arrows in the Middle
  \draw[blue, -,postaction={decorate,decoration={markings,mark=at position 0.5 with {\arrow[thick]{stealth};}}}] (dot11) -- (dot21);
  \draw[blue, -,postaction={decorate,decoration={markings,mark=at position 0.5 with {\arrow[thick]{stealth};}}}] (dot10) -- (dot20);
% Curved Line with Arrow in the Middle Top
  \draw[red, -,postaction={decorate,decoration={markings,mark=at position 0.5 with {\arrow[thick]{stealth};}}}] (dot11) to[out=35, in=145] (dot21);
% Curved Line with Arrow in the Middle Bottom
  \draw[blue, -,postaction={decorate,decoration={markings,mark=at position 0.5 with {\arrow[thick]{stealth};}}}] (dot21) to[out=215, in=-35] (dot11);
  
  % Curved Line with Arrow in the Middle Bottom
  \draw[blue, -,postaction={decorate,decoration={markings,mark=at position 0.5 with {\arrow[thick]{stealth};}}}] (dot20) to [out=215, in=-35] (dot10);
  % Curved Line with Arrow in the Middle Bottom
  \draw[red, -,postaction={decorate,decoration={markings,mark=at position 0.5 with {\arrow[thick]{stealth};}}}] (dot10) to[out=35, in=145] (dot20);

  \end{tikzpicture}
\end{minipage}%
\begin{minipage}{0.6\linewidth} % Adjust the width as needed
  \centering
  \begin{tikzpicture}
    % Your second TikZ code here
% First Ellipse
  \node[ellipse, minimum width=1.8cm, minimum height=4.8cm, draw] (ellipse1) at (0,0) {};
  \node[circle, fill, inner sep=1pt] (dot11) at (0, 1) {};
  \node[circle, fill, inner sep=1pt] (dot1-1) at (0, -1) {};

     % Loop with Arrow on the Third Dot
  \draw[blue, -,postaction={decorate,decoration={markings,mark=at position 0.5 with {\arrow[thick]{stealth};}}}] (dot11) to[out=210, in=150, looseness=140] (dot11);
  
    % Loop with Arrow on the Third Dot
  \draw[blue, -,postaction={decorate,decoration={markings,mark=at position 0.5 with {\arrow[thick]{stealth};}}}] (dot11) to[out=200, in=160, looseness=110] (dot11);
 
   % Loop with Arrow on the Third Dot
  \draw[blue, -,postaction={decorate,decoration={markings,mark=at position 0.5 with {\arrow[thick]{stealth};}}}] (dot1-1) to[out=210, in=150, looseness=140] (dot1-1);
  
    % Loop with Arrow on the Third Dot
  \draw[blue, -,postaction={decorate,decoration={markings,mark=at position 0.5 with {\arrow[thick]{stealth};}}}] (dot1-1) to[out=200, in=160, looseness=110] (dot1-1);

  \end{tikzpicture}
\end{minipage}%

\caption{  { \it The first graph is before the edge contraction and the second graph is after the edge contraction. The red edges denote the edges along which the contraction takes place. The blue edges in the first graph form loops after the contraction.} }
\end{figure}

The formula for $m_{\Omega}$ and $m_{\widehat{\Omega}}$ assigns weights to vertices and edges of the graph. Because of the vertices on the second ellipse, we get $\phi_1(i_-)\tau_{i_-}\omega_{i_-}$ as additional weights on the vertices compared to the second graph.
The weights on the blue edges cancel out and we are left with the weights on the red edges which is precisely, $\phi_1(i_-)\tau_{i_+}\omega_{i_-}$. Since, $\tau_{i_+} = \tau_{i_-}$, we get $m_{\widehat \Omega}(\tau,\omega) - m_\Omega(\tau,\omega) = -2\tau_{i_-}\omega_{i_-} (\phi_1(i_-))$.

\subsection{Compatibility of $j_!\mu_\nu^*$ with the PBW basis} 
Let $\mcO$ be a $\gvF$- orbit in $\evF$. Let $P_{\mcO}: \gvF \backslash \evF \to \{0, 1\}$ be defined as $P_{\mcO}(\mcO) = 1$ and $0$ otherwise, i.e., it is the characteristic function of the orbit $\mcO$. Since, $\evF$ is finite, the number of orbits is also finite. Let $\epsilon = \{P_{\mcO} ~|~ \mcO \in  \gvF \backslash \evF\}$. 

Now, $P_{\mcO}$ induces a map $\zeta_{\mcO}: \evF \to \mathbb{Q}(q^{1/2})$ by defining $\zeta_\mcO(x) = P_{\mcO}([x])$ where $[x]$ denotes the $\gvF$-orbit of $x$. 
The set $\{\zeta_\mcO\}_{\mcO \in \gvF \backslash \evF}$ forms the PBW basis of $\hO$.\\

{\bf Remark:} Let $X$ and $Y$ be finite sets with a group action by $G$. A map on sets, say $\phi: X \to Y$ induces a well defined map on $\mu: G\backslash X \to G \backslash Y$ by defining $\mu([x]) = [\phi(x)]$ if $\phi$ preserves the group action, i.e., $\phi(g.x) = g.\phi(x)$.\\

Since $\mu_{\nu}: \evhF \to \evhatF$ defined by $\mu_{\nu}(x) = \widehat{x}$ is compatible with the group action $\gvF$, by the above remark, we have a bijection $\mu_{\nu}: \gvF \backslash \evhF \to \gvhatF \backslash \evhatF$ which is defined as $\mu_\nu([x]) = [\widehat{x}]$. Next we show that the map $j_!\mu_\nu^*$ preserves the PBW basis.\\

\begin{proposition}
    If $\mu_\nu(\mcO) = \mcO'$ then, $j_!\mu_\nu^*(P_{\mcO'}) = P_\mcO$.
\end{proposition}
\begin{proof} Let $x \in \evF$ such that $[x] = \mcO$ and $x \in \evhF$. This implies $\mu_\nu([x]) = \mu_\nu(\mcO) = \mcO'$. The computation is worked out as follows:
\[
\begin{aligned}
 j_!\mu_\nu^*(P_{\mcO'})([x])& = j_!(\mu_\nu^*P_{\mcO'})([x])\\
    & = \ds\sum_{x^\heart \in j_{\nu}^{-1}([x])}\mu_\nu^*(P_{\mcO'})([x^\heart])\\
    & = P_{\mcO'}\mu_\nu([x])\\
    & = P_{\mcO'}(\mcO')\\
    & = 1\\
    & = P_{\mcO}(\mcO)\\
    & = P_{\mcO}([x]).
\end{aligned}
\]
If $x \in \evF$ such that either $x \notin \evhF$ or $[x] \neq \mcO$ then $j_!\mu_\nu^*(P_{\mcO'})([x]) = 0 = P_\mcO([x])$.\\
Therefore we have, $j_!\mu_\nu^*(P_{\mcO'}) = P_\mcO$.
\end{proof}

\subsection{Split short exact sequence of Hall Algebras}

We will show that $\hOhat$ is a split subquotient of $\hO$, which was proved in \cite{Li23}.

Recall the representation space with isomorphisms on the edges along which the contaction takes place, i.e., $\evhF$. Let $\evcF$ be its compliment in $\evF$. This space is $\gvF$-stable. Just like we constructed a Hall algebra associated to the representation spaces $\evF$ and $\evhF$, we do the same for $\evcF$ as well. 
Let $\hvOc = \{f: \evcF \to \mb{Q}(q^{1/2}) ~|~ f(g.x) = f(x) ~\forall g \in \gvF, x \in \evcF \}$. 
We define,
\[ \hOc = \ds\oplus_{\nu \in \mb{N}[\widehat{I}]} \hvOc  .
\]
We further define a new space
\[ \hOIhat = \ds\oplus_{\nu \in \mb{N}[\widehat{I}]} \hvO.
\]
By definition, we have
\[
\hOIhat = \hOh \oplus \hOc.
\]

\begin{lemma}
$\hOc$ is a two-sided ideal of $\hOIhat$.
\end{lemma}
\begin{proof} We need to show that if $f \in \hOc$ and $g \in \hOIhat$ then $f \circ f \in \hOc$. With loss of generality, assume that $f \in \hwO^c$ and $g \in \htO$ where $\omega, \tau \in \whI$. So it is sufficient to prove that $f \circ g \in {\bf H}_{\tau + \omega, \Omega}^c$, i.e., to show that $f \circ g$ is a map on $\evcF$ where the dimension vector of $V = \nu = \tau + \omega$. 
Consider $(x', x'') \in \ewcF \times \etcF$ where dim$({\bfW}) = \omega$ and dim$({\bfT}) = \tau$. Suppose that one of the components in the pair $(x', x'')$ is either in $\etcF$ or $\ewcF$. Suppose $x \in \evF$, such that $x = p_3p_2p_1^{-1}(x', x'')$. Since $V = W \oplus T$, we have $x = \begin{bmatrix}
\begin{array}{c|c}
x' & x^{{\bf T} \to \bfW} \\ \hline
0 & x'' \\ 
\end{array}
\end{bmatrix}$. Therefore, $x$ is invertible iff $x'$ and $x''$ is invertible. This shows that $x \in \evcF$. Hence, $\hOc$ is a two-sided ideal of $\hOIhat$.
\end{proof}

Note that the map $j_\nu^*: \hvO \to \hvOh$ is a restriction map and is also surjective by extension by zero. Taking direct sum over $\nu \in \widehat{I}$, we define a new map,
\[
j^* = \ds\oplus_{\nu \in \whI} ~j_\nu^*: \hOIhat \to \hOh.
\]

\begin{theorem}
    We have the following split short exact sequence of Hall algebras
   \[
\begin{tikzcd}[row sep=small, column sep=large]
    0 \arrow[r] 
      & \hOc \arrow[r] 
      & \hOIhat \arrow[r, "j^*", shift left=2pt]
      & \hOh \arrow[r] 
        \arrow[l, "j_!", shift left=2 pt]
      & 0
\end{tikzcd}
\]
Furthermore, $\hOhat$ is a split subquotient of $\hO$, and $\hOhat \cong \hOIhat / \hOc$. \label{split subquotient}
\end{theorem}
\begin{proof} The proof is identical to the proof in \cite{Li23} page 17 proposition 2.4.1.
\end{proof}


\vspace{2 cm}








\begin{thebibliography}{99999}\frenchspacing

\bibitem[Hall59]{Hall59}
P. Hall, {\em The Algebra of Partitions}, Proceedings of the 4th Canadian Mathematical Congress, Banff (Toronto), University of Toronto Press, (1959), pp. 147–159.\\


\bibitem[Li23]{Li23}
Y. Li, {\em Quantum Groups and Edge Contractions}, arXiv preprint arXiv:2308.16306 (2023).\\

\bibitem[LR24]{LR24}
Y. Li and J. Ren, {\em Critical Cohomological Hall Algebra and Edge Contraction}, arXiv preprint arXiv:2401.04839 (2024).\\

\bibitem[Lus91]{Lus91}
G. Lusztig, {\em Quivers, Perverse Sheaves, and Quantized Enveloping Algebras}, Journal of the American Mathematical Society 4, no. 2 (1991): 365–421. https://doi.org/10.2307/2939279. \\

\bibitem[Lus98]{Lus98}
G. Lusztig, {\em Canonical Bases and Hall Algebras}, In: Broer, A., Daigneault, A., Sabidussi, G. (eds) Representation Theories and Algebraic Geometry, Nato ASI Series, vol 514. pages 365–399, Springer, Dordrecht. https://doi.org/10.1007/978-94-015-9131-7\_9 (1998).\\

\bibitem[Lus10]{Lus10}
G. Lusztig, {\em Introduction to Quantum Groups}, United Kingdom: Birkhäuser Boston, 2010.\\


\bibitem[Mak18]{Mak18}
R. Maksimau, {\em Categorical Representations and KLR Algebras}, Algebra Number Theory 12 (2018), no. 8, 1887--1921. \\

\bibitem[Rin90]{Rin90}
C. M. Ringel, {\em Hall Algebras and Quantum Groups}, Invent Math 101 (1990), no. 3, 583--591.\\


\bibitem[Ste01]{Ste01}
E. Steinitz, {\em Zur Theorie der Abel’schen Gruppen}, Jahresbericht der DMV 9 (1901), 80–5.\\





 \end{thebibliography}
\end{document}